\documentclass{amsart}
\usepackage{paper}
\graphicspath{{figs/}}
\title{Method of Moments for Estimation of Noisy Curves}

\author{Phillip Lo}
\thanks{(Phillip Lo) The University of Chicago, Committee on Computational and Applied Mathematics, current institution Biohub, \texttt{phillip.lo@czbiohub.org}}
\author{Yuehaw Khoo}
\thanks{\hspace*{1.15em}(Yuehaw Khoo) The University of Chicago, Department of Statistics, Committee on Computational and Applied Mathematics, \texttt{ykhoo@uchicago.edu}}

\begin{document}

\begin{abstract} 
    In this paper, we study the problem of recovering a ground truth high dimensional piecewise linear curve $C^*(t):[0, 1]\to\bbR^d$ from a high noise Gaussian point cloud with covariance $\sigma^2I$ centered around the curve. We establish that the sample complexity of recovering $C^*$ from data scales with order at least $\sigma^6$. We then show that recovery of a piecewise linear curve from the third moment is locally well-posed, and hence  $O(\sigma^6)$ samples is also sufficient for recovery. We propose methods to recover a curve from data based on a fitting to the third moment tensor with a careful initialization strategy and conduct some numerical experiments verifying the ability of our methods to recover curves. All code for our numerical experiments is publicly available on GitHub.
\end{abstract}
 
\maketitle
\vspace{-3em}
\tableofcontents  

\section{Introduction}
\subsection{Motivation and Related Work}
Manifold learning is a widely-studied problem in statistics in which a low dimensional manifold is fit to high dimensional data in an effort to understand the geometry of the data and circumvent the curse of dimensionality. A vast amount of literature has been dedicated to the noise-free case, where the data lie exactly on some lower dimensional manifold; see, for instance, \cite{meila-review} for a review of more classical methods, or \cite{feff-noiseless}, \cite{kde-noiseless} for more recent results. Work has also been done in the case of low noise, such as in \cite{manifold-unbounded-noise}, \cite{feff-small-noise}, \cite{hautieng-wu}, \cite{manifoldron}, and \cite{cheng-landa}. Less work has been done for the more difficult high noise case, in which each data point is completely dominated by noise. In \cite{feff-high-noise}, the authors establish an algorithm for recovering manifolds with particular smoothness properties corrupted by Gaussian noise with covariance $\sigma^2$ for arbitrary $\sigma$, but their algorithm requires a number of samples exponential in $\sigma^2$. An important earlier work in studying the sample complexity of manifold recovery is \cite{genovese-minimax}, in which the authors use information-theoretic techniques to show that the minimax risk of estimating a Riemannian manifold in Hausdorff distance from $N$ samples is bounded below by order $1/\log N$.

Another field in which recovering a signal from high noise has been studied is in orbit recovery problems. In these problems, an underlying signal is corrupted by some group action and a large amount of additive noise, and the goal is to recover the original signal. Perhaps the most famous such problem is the reconstruction of the 3D structure of macromolecules from cryo-electron microscopy (cryo-EM) images. In cryo-EM, a large number of extremely noisy images are taken of a molecule in various 3D orientations. We can think of these images as projections of an unknown rigid rotation of the Coulomb potential $f:\bbR^3\to\bbR$ along an axis to a two-dimensional image that is then corrupted by a large amount of Gaussian noise. The task of recovering $f$ thus involves both denoising the Gaussian noise and undoing the rotation. See \cite{singer} for a comprehensive review of computational aspects of cryo-EM. A simpler orbit recovery problem that has been well-studied is the multireference alignment (MRA) problem. In the MRA model, the ground truth is a one-dimensional signal $f\in\bbR^d$ and the measurements consist of a random cyclic permutation of $f$ corrupted with additive Gaussian noise. The sample complexity of MRA is established in \cite{mra}, where the authors prove that for large $\sigma$, any algorithm hoping to recover $f$ must require $O(\sigma^6)$ samples.

We are interested in recovering one-dimensional manifolds (i.e., curves) in high dimensions in the presence of large Gaussian noise. Some of the earliest work done in fitting curves to data comes from \cite{principal-curves}, in which the authors coin the term \textit{principal curves} for curves interpolating data; they view principal curves as a nonlinear generaliation of principal component analysis, which foreshadows our methods used in Figure \ref{fig:tracing-lohi}. In our work, for a particular class of curves, we will establish the sample complexity of the problem to be $O(\sigma^6)$ and provide a recovery algorithm that meets that lower bound. The setting of one-dimensional manifolds corresponds to temporal data, and we can think of noisy curves as noisy measurements of some system evolving in time; recovering the underlying curve thus corresponds to recovering the underlying dynamics. In the one-dimensional setting, a great deal of work has been done in the \textit{seriation} problem, where noisy observations of time-dependent data are labeled with ordered timestamps. For instance, in cryo-EM, one might have noisy observations of a biological macromolecule undergoing some conformational change over time; in this context, the seriation problem involves estimating the temporal order of these conformers (see for instance \cite{Moscovich}, \cite{Lederman}). In archaeology, seriation techniques are used to date archeological finds \cite{archaeology}. Various spectral methods have been developed for the seriation problem (\cite{spectral-seriation-1}, \cite{spectral-seriation-2}, \cite{spectral-seriation-3}), but often require strong structural assumptions on the data. More recently, the authors in \cite{yuehaw-seriation} use the Fiedler vector of the graph Laplacian of the data to approach the seriation problem in a more general case. Note that our problem is not to only recover the time ordering of noisy points coming from a curve, but to recover the underlying curve itself. If the distribution of the noise is known, then the time ordering of noisy points can be computed from the underlying curve simply by associating each noisy data point with a point in the underlying curve that maximizes the likelihood.

\subsection{Problem Formulation: Noisy Curve Recovery}
Consider a parametric, non self-intersecting ground truth piecewise linear (PWL) curve $C^*(t)$ in $\bbR^d$, $t\in[0, 1]$. We assume that the curve consists of a finite number of segments. In order to fix a scale, we assume that $\|C^*(t)\|<1$ for all $t\in[0, 1]$. The \textit{noisy curve model} consists of $N$ independent observations
\begin{equation}
y^{(j)} = C^*(\tau^{(j)})  + \sigma\xi^{(j)}, \hspace{2em}j = 1,\dots, N.
\label{eqn:noisy-curve-model}
\end{equation}
Here $\tau\sim U[0, 1]$ is drawn uniformly randomly and $\xi\sim\mathcal{N}(0, I_d)$ denotes isotropic Gaussian noise independent of $\tau$; we assume $\sigma$ is known. We can imagine this as a Gaussian ``fattening'' of the curve $C^*$. Let $P_{C^*,\sigma}$ denote the random variable given by $C^*(\tau) + \sigma\xi$, and let $P_{C^*}$ denote the case where $\sigma = 0$. We formulate the \textit{noisy curve recovery} problem as the estimation of $C^*$ from observations $\{y^{(j)}\}_{j=1}^N$, as depicted in Figure \ref{fig:problem-statement}. Note that we can think of this model as a continuous Gaussian mixture model, where the mixture centers are drawn continuously from $C^*(t)$. 
\begin{figure}[h]
\centering
\includegraphics[width=5in]{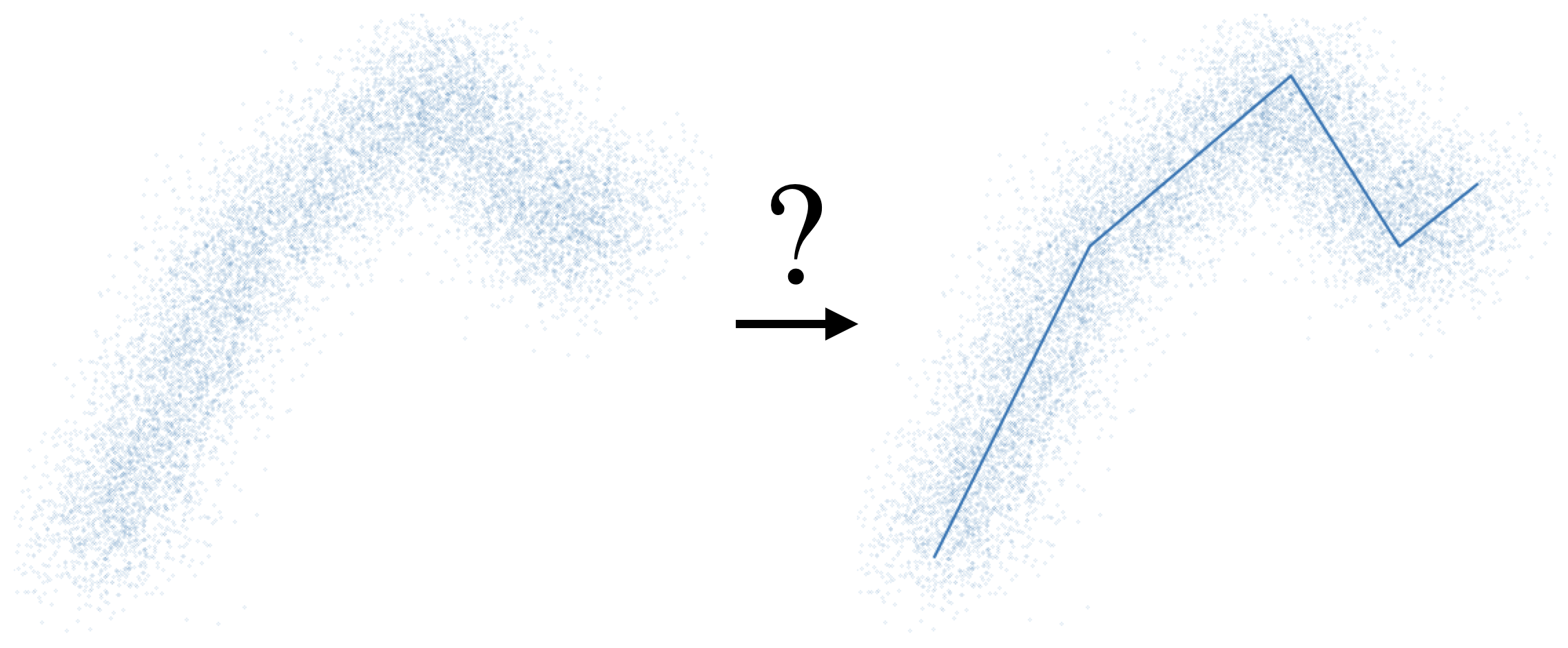}
\caption{An illustration of the noisy curve recovery problem.}
\label{fig:problem-statement}
\end{figure}

In the low noise regime, the salient features of the underlying curve are still visible in the presence of noise, and the curve can be estimated by ``tracing through'' the point cloud coming from (\ref{eqn:noisy-curve-model}). We provide a more detailed account of an algorithm in \S\ref{sec:cloud-tracing} of the supplementary material, but essentially, much like the expectation-maximization step in estimating finite Gaussian mixture models, we can assume that nearby points in the cloud come from nearby points on the underlying curve. We can thus use the local largest principal component of subsets of the point cloud to approximate the tangent space of the curve near that point. However, in the case where noise is large, we can no longer reliably assign points in the point cloud to a nearby point on the underlying curve. This is illustrated in Figure \ref{fig:tracing-lohi}, where we demonstrate the results of tracing through a point cloud coming from a PWL curve in both the low and high noise case. We are henceforth interested in studying the high noise case.
\begin{figure}
\centering
\includegraphics[width=5in]{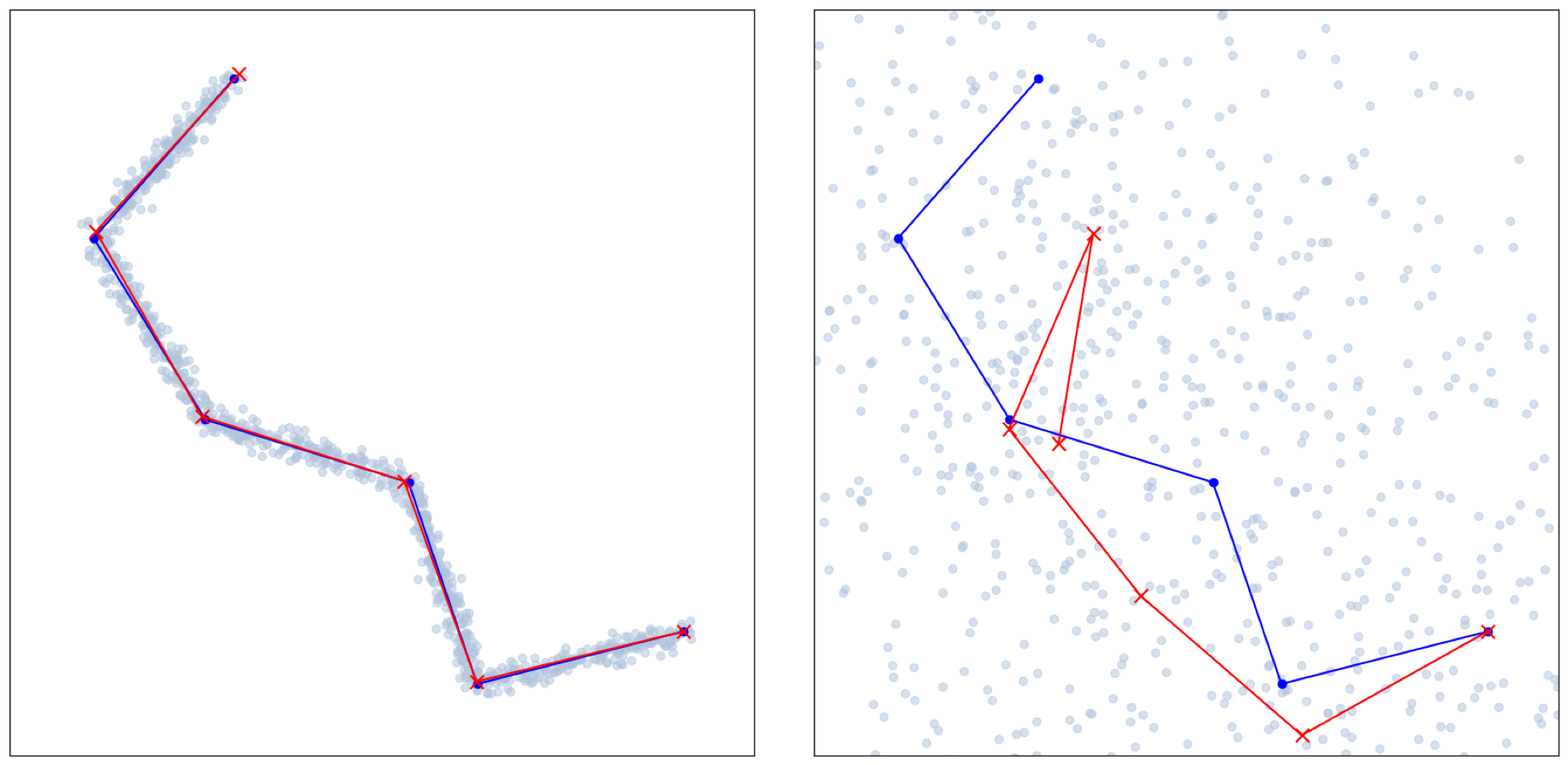}
\caption{Results of a cloud tracing algorithm (detailed in \S\ref{sec:cloud-tracing} of the supplementary material) on a cloud coming from a piecewise linear curve with low noise (left) and high noise (right). In both cases, a ground truth curve in blue is shown with an accompanying point cloud. When the noise is low (i.e., the scale of the noise is smaller than the scale of the geometric features of the curve), the structure of the underlying curve is still visible in the point cloud, and the curve can be successfully traced through (in red). When the noise is large, the local geometric structure of the curve is destroyed and the tracing approach fails.}
\label{fig:tracing-lohi}
\end{figure}

In this paper, we will first prove a result on the sample complexity of the high noise curve recovery problem, that is, the number of samples needed to recover the ground truth underlying curve from data up to some error. In particular, we will show that for sufficiently high noise, the number of samples from (\ref{eqn:noisy-curve-model}) needed to recover a curve scales with order $\sigma^6$. The proof relies on studying the moments of the distribution $P_{C^*,\sigma}$, particularly the third moment. Our approach is based largely on the approach in the proof of the sample complexity of the MRA problem in \cite{mra}; the proof hinges on showing that the first two moments alone do not uniquely determine a curve in a neighborhood. We then prove that the problem of recovering the underlying curve from the \textit{third} moment of the data is in fact locally well-posed, and provide an algorithm to obtain an initial rough estimate of the curve based on a decomposition of the third moment tensor of the curve. The idea of the algorithm is to use the tensor power method to approximate the subspaces in which vertices of the curve lie, then use this as an initialization to a gradient descent scheme to fit a curve defined by the vertices to the third moment.

\subsection{Paper Outline}
The paper is organized as follows. In \S\ref{sec:pwl-curve-moments}, we introduce notation for piecewise linear curves and compute their moments. In \S\ref{sec:sample-complexity-result}, we state and outline the proof of the sample complexity result. In \S\ref{sec:local-well-posed}, we show that locally, the third moment uniquely determines piecewise linear curves in high dimensions. In \S\ref{sec:algo}, we use this local well-posedness to devise an algorithm that recovers piecewise linear curves. We show results from numerical experiments demonstrating the performance of our algorithm in \S\ref{sec:numerical-results}.

\subsection{Tensor Notation}
Throughout this work we will use several notations for different tensor operations; we collect them all here for convenience, although we will also repeat these definitions later in the text as needed for the sake of clarity.

For a vector $v\in\bbR^d$, $v^{\otimes k}$ denotes the $k$-fold tensor power of $v$. For instance, $v^{\otimes 3}$ is the $(d, d, d)$-shaped tensor whose $jkl$ entry is given by $v_jv_kv_l$. For a three-way tensor $X$, $\Sym:\bbR^{d\times d \times d}\to\bbR^{d\times d\times d}$ is the symmetrization operator on three-way tensors, i.e., 
\begin{equation}
    \label{eqn:sym-def}
    \Sym(X)_{jkl} = \frac{1}{6}\left(X_{jkl} + X_{klj} + X_{ljk} + X_{lkj} + X_{kjl} + X_{jlk}\right).
\end{equation}
Note that $\Sym$ is linear.

The following notations are ad-hoc and nonstandard. For $(a, b)$-shaped matrices $M, N, O$, we use $M\otimes N\otimes O$ to denote the six-way tensor $(a, b, a, b, a, b)$-shaped tensor whose $mjnkol$ entry is given by $M_{mj}N_{nk}O_{ol}$; $M^{\otimes 3}$ is shorthand for $M\otimes M\otimes M$. For a six-way tensor $X$ with shape $(a, b, a, b, a, b)$ and a three-way tensor $y$ with shape $(a, a, a)$, $X\star y$  denotes the $(b, b, b)$-shaped tensor resulting from the natural contraction between $X$ and $y$:
\[[X\star y]_{jkl} = \sum_{m,n,o = 1}^a X_{mjnkol}y_{mno}.\]

\section{Piecewise Linear Curves and their Moments}
\label{sec:pwl-curve-moments}
We restrict our analysis to piecewise linear (PWL) curves in order to have a simple explicit representation of a curve. We also require that our curves have domain $[0, 1]$, are open, and have constant speed, i.e., $\|C'(t)\|$ is constant.

Let $0 = t_0 <\dots < t_M = 1$ be points in time and let $c_0 = c(t_0), \dots, c_M = c(t_M)\in\bbR^d$ be the vertices of a PWL curve. Throughout this paper, we use $d$ to denote the ambient dimension of the curve and $M$ to denote the number of segments; we assume that both are known. We assume in general that our dimension is high enough that $d\geq M$. We also assume that $c_i\neq c_j$ for all $i\neq j$. In particular $c_0\neq c_M$, so our curve is open.

We have the following explicit expression for $C(t)$:
\begin{equation}
C(t) = \frac{t - t_{i-1}}{t_i - t_{i - 1}}c_i - \frac{t - t_i}{t_i - t_{i-1}}c_{i-1},\; \text{when }t\in [t_{i-1}, t_i]_{i = 1,\dots, M}.
\label{eqn:curve}
\end{equation}
Since we require $C(t)$ to have constant speed, we must have
\begin{equation}
t_i = \frac{\sum_{k=1}^i \|c_{k} - c_{k-1}\|}{Z}\;\;\;\textrm{for }i = 1,\dots, M
\label{eqn:arclength-param-t}
\end{equation}
where $Z \defeq \sum_{k=1}^M \|c_k - c_{i-1}\|$ is the total length of the curve. The benefit of requiring constant speed is that our curves are completely characterized by the vertices $c_0,\dots,c_M$. In a slight abuse of notation, we thus identify a PWL curve with the $(M+1,d)$-shaped matrix of its row-stacked vertices:
\begin{equation}
C = \begin{bmatrix} -c_0-\\-c_1-\\\vdots\\-c_M-\end{bmatrix}.
\label{eqn:C-mat}
\end{equation}
We assume that the $c_i$ are in generic position, so that the vectors $\{c_i - c_{i-1}\}_{i=1}^M$ are linearly independent.

Our approach to proving lower bounds on the sample complexity of noisy curve recovery is based on studying the moments of a curve, defined by
\begin{equation}
m_k(C) = \int_0^1 C(t)^{\otimes k}\,dt.
\label{eqn:curve-mom}
\end{equation}
Here, $\cdot^{\otimes k}$ denotes the $k$-fold tensor power of a vector. For instance, for $v\in\bbR^d$, $v^{\otimes 3}$ is the $(d, d, d)$-shaped tensor whose $jkl$ entry is given by $v_jv_kv_l$. The use of moment tensors for inverse problems has been studied in many other contexts, such as finite Gaussian mixture models in \cite{gmm-mom}.

We now derive expressions for the first three moments of $C$ in terms of the vertices $c_0,\dots, c_M$.
\begin{prop}[Moments of Piecewise Linear Curves]
\label{prop:moments}
Let $C$ be the constant speed PWL curve with domain $[0, 1]$ with vertices $c_0,\dots, c_M$. Let $Z$ be the length of $C$. Then the first three moments of $C$ are given by the formulas
\begin{align}
    \begin{split}
    m_1(C) &= \frac{1}{2}\sum_{i=1}^M\frac{\|c_i - c_{i-1}\|}{Z}(c_{i-1} + c_i)\\
    m_2(C) &= \frac{1}{6}\sum_{i=1}^M \frac{\|c_i - c_{i-1}\|}{Z}\left((c_{i-1} + c_i)^{\otimes 2} + c_{i-1}^{\otimes 2} + c_i^{\otimes 2}\right)\\
    m_3(C) &= \frac{1}{12}\sum_{i=1}^M\frac{\|c_i - c_{i-1}\|}{Z}\left((c_{i-1}+c_i)^{\otimes 3} + 2c_{i-1}^{\otimes 3} + 2c_i^{\otimes 3}\right).
    \end{split}
    \label{eqn:clean-curve-moments}
\end{align}
\end{prop}
\begin{proof}
Consider the segment-wise change of variables 
\begin{equation}
    \frac{t - t_{i-1}}{t_i - t_{i-1}}\mapsto t, \hspace{2em} dt\mapsto (t_i - t_{i-1})dt
    \label{eqn:change-of-vars}
\end{equation}
For any $k$, we can apply this change of variables and (\ref{eqn:arclength-param-t}) to get an expression for $m_k(C)$ in terms of the vertices $c_i$:
\begin{align} 
    \begin{split}
    m_k(C) &= \int_0^1 C(t)^{\otimes k}\,dt\\
    &= \sum_{i=1}^M\int_{t_{i-1}}^{t_i} \left(\frac{t - t_{i-1}}{t_i - t_{i - 1}}c_i - \frac{t - t_i}{t_i - t_{i-1}}c_{i-1}\right)^{\otimes k}\,dt\\ 
    &= \sum_{i=1}^M\int_0^1\left(tc_i + (1-t)c_{i-1}\right)^{\otimes k}(t_i - t_{i-1})\,dt\\
    &= \sum_{i=1}^M\frac{\|c_i - c_{i-1}\|}{Z}\int_0^1((1-t)c_{i-1} + tc_i)^{\otimes k}\,dt.
    \end{split}
\end{align}
The desired result can be obtained by expanding the tensor power in the integral for $k = 1, 2, 3$ and direct computation of the integral.
\end{proof}

Recall our assumption that $d\geq M$ and observe that a PWL curve $C$ with $M$ segments naturally lives in an $M$-dimensional subspace $V$ of $\bbR^n$ spanned by the vectors $\{c_i - c_{i-1}\}_{i=1}^M$. We can in fact compute an orthonormal basis for this subspace by taking the top $M$ eigenvectors of the second moment $m_2(C)$; indeed, the second moment is a rank-$M$ matrix for generic $C$. Note that the second moment $m_2(C)$ can be estimated with $O(\sigma^4)$ samples from a noisy curve using the unbiased estimator in Lemma \ref{lem:unbiased-estimators}. Therefore we will often assume $M=d$, since if $d>M$, we can project the coordinates of $C$ down to $V$. Nonetheless, for semantic clarity, we will often use $M$ when referring to the number of segments and $d$ when talking about the dimension.

Later on in the paper, we will show two important results involving moment matching. We will show that matching a PWL curve to a third moment $m_3(C)$ is locally well-posed, while matching a PWL curve to the first two moments $m_1(C)$ and $m_2(C)$ is locally ill-posed. A careful analysis of the moments is difficult owing to the $\|c_i - c_{i-1}\|/Z$ terms (note that $Z$ also depends on $C$). To aid in our analysis and make things easier to compute, we introduce the following \textit{relaxed} curve moments that replace the $\|c_i - c_{i-1}\|/Z$ terms with new variables. We will show that it is sufficient to consider moment matching in this relaxed setting.
\begin{defn}[Relaxed Curve Moments]
Let $C\in\bbR^{(M+1)\times d}$ and $p\in\bbR^M$. Then we define the \textit{relaxed moments} of $C$ to be
\begin{align}
    \begin{split}
    \mu_{1}(C, p) &= \frac{1}{2}\sum_{i=1}^M p_i(c_{i-1} + c_i)\\
    \mu_{2}(C, p) &= \frac{1}{6}\sum_{i=1}^M  p_i\left((c_{i-1} + c_i)^{\otimes 2} + c_{i-1}^{\otimes 2} + c_i^{\otimes 2}\right)\\
    \mu_{3}(C, p) &= \frac{1}{12}\sum_{i=1}^M p_i\left((c_{i-1}+c_i)^{\otimes 3} + 2c_{i-1}^{\otimes 3} + 2c_i^{\otimes 3}\right).
    \end{split}
\end{align}
We can rewrite these in a more vectorized way that makes our analysis more straightforward.
\begin{align}
    \begin{split}
    \mu_{1}(C, p) &= \frac{1}{2}C^TA^sp\\
    \mu_{2}(C, p) &= C^T\ol{A}(p)C\\
    \mu_{3}(C, p) &= C^{\otimes 3}\star\alpha(p).
    \end{split}
    \label{eqn:decoupled-moments}
\end{align}
Here, $A^s$ is an $(M+1,M)$-shaped matrix constant and $\ol{A}(p)$ is a symmetric $(M+1, M+1)$-shaped matrix that is a linear function of $p$. The six-way tensor $C^{\otimes 3}$ has shape $(M+1, d, M+1, d, M+1, d)$ and has $mjnkol$ entry given by $C_{mj}C_{nk}C_{ol}$. The three-way tensor $\alpha(p)$ is a linear function of $p$ and is symmetric with shape $(M+1, M+1, M+1)$. The operation $\star$ denotes the natural tensor contraction between tensors of shape $(M+1, d, M+1, d, M+1, d)$ and $(M+1, M+1, M+1)$:
\begin{equation}
    [X\star y ]_{jkl} = \sum_{m,n,o = 0}^M X_{mjnkol}y_{mno}.
    \label{eqn:star-contraction}
\end{equation}
The exact forms of $A^s$, $\ol{A}(p)$, and $\alpha(p)$ are given in \S\ref{sec:custom-notation} of the supplementary material.
\label{def:decoupled-curve-moments}
Note that in the case where $p$ is the vector of proportional segment lengths of $C$, i.e.,  $p_i = \|c_i - c_{i-1}\|/Z$ for all $i = 1,\dots, M$, the relaxed moments are precisely equal to the actual moments.
\end{defn}

We now present expressions for the Jacobians of the error between a predicted relaxed third moment and a ground truth third moment. The proof is by standard matrix calculus techniques, which we omit.
\begin{lem}[Jacobians of Relaxed Moments]
Let $C\in\bbR^{(M+1)\times d}$ be the matrix representation of a PWL curve and $p\in\bbR^M$; note that we do not require $p$ to be equal to the proportional segment lengths of $C$. Let $m_3\in\bbR^{d\times d \times d}$ be some third moment tensor. Define the third moment loss
\begin{equation}
    L_{m_3}(C, p) = \|\mu_{3}(C, p) - m_3\|_F^2,
    \label{eqn:loss-defns} 
\end{equation}
where $\|\cdot\|_F^2$ is the squared Frobenius norm on tensors, i.e., the sum of the squared entries.  Let $y\in\bbR^{(M+1)\times d}$ and $z\in\bbR^{M}$. Let $\alpha(p)$ be as in (\ref{eqn:third-mom-alpha}) and $\star$ be as in (\ref{eqn:star-contraction}). Then the derivatives of the third moment loss with respect to $C$ and $p$ applied to $y$ and $z$ respectively are given by
\begin{align}
    \begin{split}
    \inner{\grad_C L_{m_3}(C, p)}{y} &= \Bigg\langle 2\left[C^{\otimes 3}\star\alpha(p)-m_3\right],(C\otimes C\otimes y + C\otimes y\otimes C + y\otimes C\otimes C)\star\alpha(p)\Bigg\rangle\\
    \inner{\grad_p L_{m_3}(C, p)}{z} &= \Bigg\langle 2\left[C^{\otimes 3}\star\alpha(p)-m_3\right], C^{\otimes 3}\star \alpha(z)\Bigg\rangle. 
    \end{split}
    \label{eqn:loss-jacs-applied}
\end{align}
\label{lem:loss-jacs}
\end{lem}

\section{The Sample Complexity Lower Bound}
\label{sec:sample-complexity-result}
To talk about the sample complexity of recovering PWL curves, we need a measure of the distance between two curves. We choose the following natural distance.
\begin{notn}
\label{notn:curve-dist}
Let $C(t)$ and $\Gamma(t)$ be two non-closed parametric curves on $[0, 1]$. Then we define the distance between the two curves to be the average squared Euclidean distance between the curves:
\begin{equation}
    \label{eqn:curve-dist}
    \rho(C, \Gamma) = \int_0^1 \|C(t)-\Gamma(t)\|^2\,dt.
\end{equation}
We assume that $\Gamma(t)$ is parameterized so that $\rho(C(t),\Gamma(t)) \leq \rho(C(t), \Gamma(1-t))$; i.e., $\Gamma$ is oriented in the direction that minimizes the distance.
\end{notn}

We first prove that curve recovery from the first and second moment alone is not locally well-posed. In other words, we can find two curves that are arbitrarily close together with exactly the same first and second moment. As discussed in \S{\ref{sec:pwl-curve-moments}}, we assume that the number of segments $M$ is equal to the ambient dimension $d$.
\begin{prop}
\label{prop:matching-m1m2}
Let $C^*(t):[0, 1]\to\bbR^d$ be a PWL constant speed curve with $M = d$ segments. For almost all such $C^*$, for sufficiently small $\ep>0$, there exists a PWL curve $\Gamma(t)\neq C^*(t)$ with $m_1(C^*) = m_1(\Gamma)$ and $m_2(C^*) = m_2(\Gamma)$ such that $\|C^* - \Gamma\|_F = \ep$, where the norm here is on the matrix representation of the curves.
\end{prop}
The detailed proof is provided in \S\ref{sec:matching-m1m2-proof} of the supplementary material; the idea is that matching the first and second moment of a curve can be expressed as an underdetermined system of polynomials whose solutions form a manifold of positive dimension.

Since $C^*$ and $\Gamma$ being close as matrices implies that $C^*(t)$ and $\Gamma(t)$ are close as curves, we have the following immediate corollary of Proposition \ref{prop:matching-m1m2}.
\begin{cor}
\label{cor:matching-m1m2-curve}
As in Proposition \ref{prop:matching-m1m2}, let $C^*(t):[0, 1]\to\bbR^d$ be a PWL constant speed curve with $M = d$ segments. For almost all such $C^*$, for sufficiently small $\ep>0$, there exists a PWL curve $\Gamma(t)\neq C^*(t)$ with $m_1(C^*) = m_1(\Gamma)$ and $m_2(C^*) = m_2(\Gamma)$ such that $\|C^*(t) - \Gamma(t)\| \leq \ep$ for all $t\in[0, 1]$.
\end{cor}
The application of Corollary \ref{cor:matching-m1m2-curve} is to show that the hypotheses of the following result are not vacuous for $\ell = 3$. The proof is rather technical and left to \S\ref{sec:main-chi2-bd-proof} of the supplementary material.
\begin{prop}
\label{prop:main-chi2-bd}
Let $C$, $\Gamma$ be mean zero parametric curves such that $\|C(t) - \Gamma(t)\|\leq 1/3$  and $\|C(t)\|,\|\Gamma(t)\|<1$ for all $t\in[0,1]$. Suppose $m_k(C)= m_k(\Gamma)$ for all $k \leq \ell - 1$. Then for $\sigma \geq 1$, we have the following bound on the $\chi^2$ divergence between $P_{C,\sigma}$ and $P_{\Gamma,\sigma}$:
\begin{equation}
    \label{eqn:main-chi2-bd}
    \chi^2(P_{C,\sigma}\|P_{\Gamma,\sigma})\leq \calK_\ell\sigma^{-2\ell}\rho(C,\Gamma),
\end{equation}
where $\calK_\ell$ is a constant that depends on $\ell$ only.
\end{prop}

This result allows us to bound the $\chi^2$ divergence between two noisy curve distributions by $\sigma^{-2\ell}$ if the moments of the underlying curves match up to (but not including) order $\ell$. This is the primary ingredient in the proof of the following main sample complexity result. Note that we use $\lesssim$ to denote inequality up to a universal constant.
\begin{thm}
\label{thm:main-thm}
Let $C$ be a curve as in be as in Proposition \ref{prop:main-chi2-bd}. Given any $\Delta>0$ sufficiently small, there exists another curve $\Gamma\neq C$ satisfying the hypotheses of Proposition \ref{prop:main-chi2-bd} such that $\rho(C,\Gamma) = \Delta$, where $\Gamma$ is indistinguishable from $C$ in the following sense. Consider $N$ samples $Y \defeq \{y^{(1)},\dots, y^{(N)}\}$ and the task of deciding if $Y$ came from $P_{C,\sigma}$ or $P_{\Gamma,\sigma}$. Assume a uniform prior on whether the ground truth curve is $C$ or $\Gamma$. Then if $N\lesssim \sigma^{6}\Delta^{-1}$, the probability of making an error is bounded below by $1/4$.
\end{thm}
\begin{proof}
Given $C(t)$, construct $\Gamma(t)$ such that the first two moments match and $\rho(C,\Gamma) = \Delta$ as in Proposition \ref{prop:matching-m1m2}. Then by Theorem \ref{prop:main-chi2-bd}, we have $\chi^2(P_{C,\sigma}\|P_{\Gamma,\sigma})\leq \mathcal{K}_3\sigma^{-6}\rho(C,\Gamma)$.  By Lemma \ref{lem:KL-leq-chi2}, Lemma \ref{lem:pinsker}, and Lemma \ref{lem:KLofprod},
\begin{equation}
    \TV(P_{C,\sigma}^{\otimes N}, P_{\Gamma,\sigma}^{\otimes N})^2 \leq \frac{1}{2}\KL(P_{C,\sigma}^{\otimes N}\| P_{\Gamma,\sigma}^{\otimes N})\leq \frac{1}{2}\mathcal{K}_3\sigma^{-6}\Delta N.
\end{equation}
Therefore, for $N\leq \frac{1}{2}\mathcal{K}_3^{-1}\sigma^{6}\Delta^{-1}$, we have 
\begin{equation}
    \TV(P_{C,\sigma}^{\otimes N}, P_{\Gamma,\sigma}^{\otimes N}) \leq \frac{1}{2}.
\end{equation}
Now let $\psi:\bbR^{d\times N}\to\{C, \Gamma\}$ be any measurable function of the data $Y$. Let $A$ denote the set of all $Y$ such that $\psi(Y) = \Gamma$. Then 
\begin{align}
\begin{split}
    P^{\otimes N}_{C,\sigma}(\psi(Y) = \Gamma) + P^{\otimes N}_{\Gamma,\sigma}(\psi(Y) = C) &=  P^{\otimes N}_{C,\sigma}(A) + P^{\otimes N}_{\Gamma,\sigma}(A^c)\\
    &= 1 - (P^{\otimes N}_{\Gamma,\sigma}(A) - P^{\otimes N}_{C,\sigma}(A))\\
    &\geq 1 - \sup_A (P^{\otimes N}_{\Gamma,\sigma}(A) - P^{\otimes N}_{C,\sigma}(A))\\
    &= 1 - \TV(P^{\otimes N}_{\Gamma,\sigma}, P^{\otimes N}_{C,\sigma})\\
    &\geq 1/2.
\end{split}
\label{eqn:lecam-lb}
\end{align}
Let $X$ denote the event that the ground truth curve is $C$ and $X^c$ denote the event that the ground truth curve is $\Gamma$. By our uniform prior assumption, we have $\bbP(X) = \bbP(X^c) = 1/2$. Then the probability of the hypothesis test $\psi$ making an error is 
\begin{align}
    \begin{split}
    \bbP(\psi(Y)=\Gamma | X)\bbP(X) &+ \bbP(\psi(Y) = C | X^c)\bbP(X^c) \\
    &= \frac{1}{2}P^{\otimes N}_{C,\sigma}(\psi(Y) = \Gamma) + \frac{1}{2}P^{\otimes N}_{\Gamma,\sigma}(\psi(Y) = C)\\
    &\geq \frac{1}{4}.
    \end{split}
    \label{eqn:error-prob}
\end{align}
\end{proof}
This result reveals a fundamental limitation of the curve recovery problem in the high noise regime; regardless of the method used, it is impossible to discern with high probability which of two $\Delta$-separated curves the data $Y$ come from without at least $O(\sigma^6)$ samples.

\section{Local Well-Posedness of Third Moment-Based Recovery of Piecewise Linear Curves}
\label{sec:local-well-posed}
Theorem \ref{thm:main-thm} gives us a lower bound of $O(\sigma^6)$ on the sample complexity of recovering a curve from data. We now show that for PWL curves with constant speed, this lower bound is asymptotically tight by showing that $O(\sigma^6)$ is indeed enough samples for recovery. We accomplish this by showing that the third moment of a noise-free PWL curve uniquely determines a PWL curve in a neighborhood. This is contrast to the result of Proposition \ref{prop:matching-m1m2}, which says that determining a curve from the first two moments alone is locally ill-posed. As before, in this section we assume $M=d$.

Our ultimate goal in this section is to show the following.

\begin{thm}
Let $C^*\in\bbR^{(M+1)\times d}$ be a ground truth PWL curve with proportional segment lengths $p^*\in\bbR^{M}$, where $M=d\geq 4$. Let $m_3^* = \mu_3(C^*, p^*)$ be the ground truth third moment. Let $L_{m_3^*}$ be the third moment squared Frobenius loss as defined in (\ref{eqn:loss-defns}) Then for almost all $C^*$, if $C$ is sufficiently close to $C^*$ and $p$ is sufficiently close to $p^*$, we have
\begin{align}
    \inner{\grad_C L_{m_3^*}(C, p)}{C^*-C}&<0\label{eqn:strong-C-condition}\\
    \inner{\grad_p L_{m_3^*}(C, p)}{p^*-p}&<0\label{eqn:strong-p-condition}.
\end{align}
\label{thm:strong-gradient-condition}
\end{thm}
Local well-posedness of recovery from $C$ is an immediate corollary of this.
\begin{cor}
\label{cor:m3-well-posed}
Let $C^*$, $p^*$, and $m_3^*$ as in Theorem \ref{thm:strong-gradient-condition}. Then for almost all $C^*$, there exists a neighborhood $U$ of $(C^*, p^*)$ such that $C^*$ is the only PWL curve in $U$ with third moment $m_3^*$. 
\end{cor}
\begin{proof}
The result of Theorem \ref{thm:strong-gradient-condition} implies that the ground truth $(C^*, p^*)$ is an isolated minimum of $L_{m_3^*}(C, p)$.
\end{proof}

To prove Theorem \ref{thm:strong-gradient-condition}, we will first need the following weaker result.
\begin{prop}
Let $C^*\in\bbR^{(M+1)\times d}$ be a ground truth PWL curve with proportional segment lengths $p^*\in\bbR^{M}$, where $M=d\geq 4$. Let $m_3^* = \mu_3(C^*, p^*)$ be the ground truth third moment. Then for almost all $C^*$, for $C$ sufficiently close to $C^*$ and $p$ sufficiently close to $p^*$, we have
\begin{align}
    \inner{\grad_C L_{m_3^*}(C, p^*)}{C^*-C}&<0 \label{eqn:weak-C-condition}\\
    \inner{\grad_p L_{m_3^*}(C^*, p)}{p^*-p}&<0 \label{eqn:weak-p-condition}.
\end{align}
\label{prop:weak-gradient-condition}
\end{prop}
Equation (\ref{eqn:weak-C-condition}) differs from (\ref{eqn:strong-C-condition}) in that we have replaced the $p$ in $L_{m_3^*}(C, p)$ with $p^*$, and similarly for (\ref{eqn:weak-p-condition}) and (\ref{eqn:strong-p-condition}). In other words, (\ref{eqn:weak-C-condition}) says that recovery of $C$ from $m_3^*$ and $p^*$ is locally well-posed, while (\ref{eqn:weak-p-condition}) says that recovery of $p$ from $m_3^*$ and $C^*$ is locally well-posed. On the other hand, the stronger statements (\ref{eqn:strong-C-condition}) and (\ref{eqn:strong-p-condition}) say that recovery of $C$ and $p$ from $m_3^*$ is locally well-posed. The proof of Proposition \ref{prop:weak-gradient-condition} is left to \S\ref{sec:weak-gradient-condition-proofs}. We now prove Theorem \ref{thm:strong-gradient-condition}.

\begin{proof}[Proof of Theorem \ref{thm:strong-gradient-condition}]
We first prove (\ref{eqn:strong-C-condition}), which we repeat here; we wish to show that for $C$ sufficiently close to $C^*$ and $p$ sufficiently close to $p^*$, we have
\[\inner{\grad_C L_{m_3^*}(C, p)}{C^*-C}<0.\]
For notational convenience, let us introduce the shorthand
\begin{equation}
    \llbracket C, C^*\rrbracket\defeq\left[C\otimes C\otimes C^* + C\otimes C^* \otimes C + C^*\otimes C\otimes C - 3C^{\otimes 3}\right].
\label{eqn:triple-C-shorthand}
\end{equation}
We wish to write the expression in (\ref{eqn:strong-C-condition}) in terms of (\ref{eqn:weak-C-condition}) plus an error term. Using (\ref{eqn:loss-jacs-applied}) and the fact $m_3^* = {C^*}^{\otimes 3}\star\alpha(p^*)$ to compute the difference and regrouping terms, we get
\begin{align}
\begin{split}
\inner{\grad_C L_{m_3^*}(C, p)}{C^*-C} &= \inner{\grad_C L_{m_3^*}(C, p^*)}{C^*-C}\\
&\hspace{2em} + 2\inner{(C^{\otimes 3} - {C^*}^{\otimes 3})\star\alpha(p^*)}{\llbracket C, C^*\rrbracket\star\alpha(p - p^*)}\\
&\hspace{2em} + 2\inner{C^{\otimes 3}\star\alpha(p - p^*)}{\llbracket C, C^*\rrbracket\star\alpha(p^*)}\\
&\hspace{2em} + 2\inner{C^{\otimes 3}\star\alpha(p - p^*)}{\llbracket C, C^*\rrbracket\star\alpha(p - p^*)}.
\end{split}
\end{align}
The first term on the right hand side above is precisely the term in (\ref{eqn:weak-C-condition}), which we showed was strictly negative for $C$ close to $C^*$ in \S\ref{sec:C-from-pstar}. By linearity of $\alpha$, the remaining three terms on the right hand side go to zero as $p\to p^*$, and hence can be made arbitrarily small for $p$ close to $p^*$ by continuity. Hence, for $(C, p)$ sufficiently close to $(C^*, p^*)$, the entire right hand side is strictly negative, and we have (\ref{eqn:strong-C-condition}); we denote the neighborhood around $(C^*, p^*)$ for which (\ref{eqn:strong-C-condition}) holds with $U_C$.

Now we prove (\ref{eqn:strong-p-condition}), which we also repeat; we wish to show that for $C$ sufficiently close to $C^*$ and $p$ sufficiently close to $p^*$, we have
\[\inner{\grad_p L_{m_3^*}(C, p)}{p^*-p}<0.\]
Similarly to before, now we wish to write the expression in (\ref{eqn:strong-p-condition}) in terms of (\ref{eqn:weak-C-condition}) plus an error term. Again using (\ref{eqn:loss-jacs-applied}) and regrouping terms, we get
\begin{align}
\begin{split}
\inner{\grad_p L_{m_3^*}(C, p)}{p^*-p} &= \inner{\grad_p L_{m_3^*}(C^*, p)}{p^*-p} \\
&\hspace{1em} + 2\inner{(C^{\otimes 3} - {C^*}^{\otimes 3})\star\alpha(p)}{C^{\otimes 3}\star\alpha(p^*)}\\
&\hspace{1em} + 2\inner{(C^{\otimes 3} - {C^*}^{\otimes 3})\star\alpha(p^*)}{{C^*}^{\otimes 3}\star\alpha(p)}\\
&\hspace{1em} + 2\inner{({C^*}^{\otimes 3} - C^{\otimes 3})\star\alpha(p^*)}{{C^*}^{\otimes 3}\star\alpha(p^*)}\\
&\hspace{1em} + 2\inner{({C^*}^{\otimes 3} - C^{\otimes 3})\star\alpha(p)}{{C^*}^{\otimes 3}\star\alpha(p)}\\
&\hspace{1em} + 2\inner{({C^*}^{\otimes 3} - C^{\otimes 3})\star\alpha(p)}{C^{\otimes 3}\star\alpha(p)}\\
&\hspace{1em} + 2\inner{({C}^{\otimes 3} - {C^*}^{\otimes 3})\star\alpha(p)}{{C^*}^{\otimes 3}\star\alpha(p^*)}.
\end{split}
\end{align}
Once again, the first term on the right hand side is precisely the term in (\ref{eqn:weak-p-condition}), which we showed was strictly negative for all $p\neq p^*$ in \S\ref{sec:p-from-Cstar}. By continuity of the function $\cdot^{\otimes 3}$, the remaining six terms on the right hand side go to zero as $C\to C^*$, and hence can be made arbitrarily small for $C$ close to $C^*$. Therefore there exists a neighborhood $U_p$ of $(C^*, p^*)$ where (\ref{eqn:strong-p-condition}) holds. 

To complete the proof of Theorem \ref{thm:strong-gradient-condition}, let $U = U_C\cap U_p$; in this neighborhood, (\ref{eqn:strong-C-condition}) and (\ref{eqn:strong-p-condition}) hold. This completes the proof of Theorem \ref{thm:strong-gradient-condition}, and thus the proof of local well-posedness of third moment-based curve recovery for piecewise linear curves.
\end{proof}

\section{A Third Moment-Based Recovery Algorithm}
\label{sec:algo}
Now we turn our attention to devising an algorithm to recover PWL curves from data. We assume that we have noisy observations coming from (\ref{eqn:noisy-curve-model}) of a ground truth PWL curve $C^*$ in $\bbR^d$ with $M$ segments; we assume that $M$ is known. We furthermore assume that $d\geq M\geq 4$, that $C^*$ is mean zero, and that the points $c^*_0,\dots,c^*_M$ are in generic position.

Theorem \ref{thm:strong-gradient-condition} suggests that a reasonable approach is to first obtain a rough estimate $\hat{C}_{\init}$ with proportional segment lengths $\hat{p}_{\init}$ in a neighborhood near the ground truth $C^*$, then perform gradient descent in $C$ and $p$ on the relaxed third moment loss (\ref{eqn:loss-defns}), using $\hat{C}_{\init}$ and $\hat{p}_{\init}$ as initializations. We propose an algorithm that does exactly this, where the initial rough estimate is obtained via the tensor power method. Note however that applying Theorem \ref{thm:strong-gradient-condition} requires knowledge of the noise-free third moment of a curve. We therefore need a method of estimating noise-free moments $m_k(C)$ from data.

In  \S\ref{sec:unbiased-estimators}, we derive unbiased estimators of the noise-free curve moments from data. In  \S\ref{sec:tpm}, \S\ref{sec:reordering}, and \S\ref{sec:final-optim}, we describe our recovery algorithm; in \S\ref{sec:tpm}, we discuss how to use the tensor power method to estimate an unordered list of subspaces that the vertices of $C^*$ live on; in \S\ref{sec:reordering}, we discuss how to reorder these subspaces; in \S\ref{sec:final-optim}, we use these reordered subspaces to obtain initial estimates $\hat{C}_{\init}$ and $\hat{p}_{\init}$ as initializations for a gradient descent scheme.

Since our algorithm requires getting good initializations for gradient descent that lie in the basin of attraction $U$ of Corollary \ref{cor:m3-well-posed}, we briefly numerically investigate the size of this basin of attraction in \S\ref{sec:basin-of-attraction} of the supplementary material.

\subsection{Unbiased Estimators of Noise-Free Moments}
\label{sec:unbiased-estimators}
Our task is to recover a ground truth curve $C^*$ from observations $y^{(j)} = C^*(\tau^{(j)})  + \sigma\xi^{(j)}$, $j = 1,\dots, N$, where $\tau\sim U[0, 1]$ and $\xi\sim\mathcal{N}(0, I)$. We want a way of recovering the noise-free moments of the curve $m_k^* \defeq m_k(C^*)$ from data. With $Y \defeq \{y^{(1)},\dots,y^{(N)}\}$, we define the $k$th empirical moment of the data $Y$ to be 
\begin{equation}
\widehat{m}_k = \widehat{m}_k(Y) \defeq \frac{1}{N}\sum_{j=1}^N (y^{(j)})^{\otimes k}.
\label{eqn:emp-mom}  
\end{equation}
First we recall the first three uncentered moments of a multivariate Gaussian.
\begin{lem}(adapted from \cite{holmquist})
Let $X\sim \calN(\mu, \Sigma)$ be a multivariate normal. Then the first through third uncentered moments of $X$ are given by
\begin{align}
    \begin{split}
    m_1(X) &= \mu\\
    m_2(X) &= \mu^{\otimes 2} + \Sigma\\
    m_3(X) &= \Sym(\mu^{\otimes 3} + 3\mu\otimes\Sigma)
    \end{split}
    \label{eqn:normal-moments}
\end{align}
Here, we use $\mu\otimes\Sigma$ to denote the three-way tensor whose $jkl$ entry is $\mu_j\Sigma_{kl}$.
\label{lem:normal-moments}
\end{lem}
From Lemma \ref{lem:normal-moments}, we can compute unbiased estimators for the noise-free moments.
\begin{lem}
\label{lem:unbiased-estimators}
Let $C$ be such that $m_1(C) = 0$ and let $y\sim C(\tau) + \sigma\xi$, where $\tau\sim U[0, 1]$ and $\xi\sim\mathcal{N}(0, I)$. Then
\begin{enumerate}
\item $y$ is an unbiased estimator of $m_1(C)$;
\item $y^{\otimes 2} - \sigma^2I_d$ is an unbiased estimator of $m_2(C)$;
\item $y^{\otimes 3}$ is an unbiased estimator of $m_3(C)$.
\end{enumerate}
\end{lem}
\begin{proof}
For the first moment:
\begin{align}
\begin{split}
    \bbE_{\tau\sim U[0, 1], \xi\sim\calN(0,I_d)}\left[y\right] &= \bbE_{\tau\sim U[0, 1]}[C(\tau)] + \bbE_{\xi\sim \calN(0, I_d)}[\sigma\xi]\\
    &= \bbE_{\tau\sim U[0, 1]}[C(\tau)]\\
    &= m_1(C).
\end{split}
\end{align}
For the second moment, eliminating cross terms by independence and the fact that $C$ and $\xi$ are mean zero, we can compute
\begin{align}
    \begin{split}
    \bbE&_{\tau\sim U[0, 1], \xi\sim\calN(0,I_d)}\left[y^{\otimes 2}\right] \\
    &= \bbE_{\tau}[C(\tau)^{\otimes 2}] + \bbE_{\tau, \xi}[C(\tau)\otimes \sigma\xi] + \bbE_{\tau, \xi}[\sigma\xi\otimes C(\tau)] + \bbE_{\xi}[\sigma\xi\otimes \sigma\xi]\\
    &= \bbE_{\tau}[C(\tau)^{\otimes 2}] + \bbE_{\xi}[\sigma\xi\otimes \sigma\xi]\\
    &= m_2(C) + \sigma^2I_d.
    \end{split}
\end{align}
For the third moment, eliminating terms with the same idea, we have
\begin{align}
    \begin{split}
    \bbE&_{\tau\sim U[0, 1],\xi\sim\calN(0,I_d)}\left[y^{\otimes 3}\right] \\
    &= \bbE_{\tau}[C(\tau)^{\otimes 3}] + \bbE_{\tau,\xi}[C(\tau)\otimes \sigma\xi \otimes \sigma\xi]+\bbE_{\tau,\xi}[C(\tau)\otimes C(\tau) \otimes \sigma\xi] \\
    &\hspace{2em}+\bbE_{\tau,\xi}[\sigma\xi\otimes C(\tau) \otimes C(\tau)] + \bbE_{\tau,\xi}[C(\tau)\otimes C(\tau) \otimes \sigma\xi] \\
    &\hspace{2em}+ \bbE_{\tau,\xi}[C(\tau)\otimes \sigma\xi \otimes C(\tau)] + \bbE_{\xi}[(\sigma\xi)^{\otimes 3}]\\
    &= \bbE_{\tau}[C(\tau)^{\otimes 3}]  + \bbE_{\xi}[(\sigma\xi)^{\otimes 3}]\\
    &= m_3(C).
    \end{split}
\end{align}
\end{proof}
Lemma \ref{lem:unbiased-estimators} and the law of large numbers imply that the empirical moments $\widehat{m}_1$, $\widehat{m}_2-\sigma_2 I$, and $\widehat{m}_3$ can estimate the noise-free moments $m_1^*$, $m_2^*$, and $m_3^*$ arbitrarily well (in probability) with a sufficient number of points as long as the noise level $\sigma$ is known. Therefore, we will design our algorithm for curve recovery from data assuming that we have estimates $\widehat{m}_1$, $\widehat{m}_2$, and $\widehat{m}_3$ from (\ref{eqn:emp-mom}) that are arbitrarily close to the ground truth moments $m_1^*$, $m_2^*$, and $m_3^*$ of $C^*$.

Note that for $y$ coming from a noisy curve with noise level $\sigma$, the entries of $y^{\otimes k}$ have variance on the order of $\sigma^{2k}$, so by Chebyshev's inequality, the number of points needed to estimate the first three moments is order $\sigma^6$. A third moment-based recovery algorithm thus shows that the sample complexity lower bound in Theorem \ref{thm:main-thm} is asymptotically tight.

\subsection{Recovering the Subspaces of $C^*$ with the Tensor Power Method}
\label{sec:tpm}
Let $T^3$ be a three-way tensor of the form 
\begin{equation}
T^3 = \sum_{i=1}^r \lambda_i u_i^{\otimes 3},
\label{eqn:odeco-tensor}
\end{equation}
where the vectors $u_i$ are orthonormal and the scalars $\lambda_i$ are all positive; such tensors are called \textit{orthogonally decomposable}. The tensor power method (TPM) is a robust iterative algorithm that is able to recover $\{\lambda_i, u_i\}_{i=1}^r$, up to some permutation of the indices. For $u_i$ that are unit length and independent (but not necessarily orthonormal), a whitening preprocessing step can be applied to reduce the problem to orthogonally decomposable case through an invertible linear transformation. A thorough introduction to the tensor power method as well as robustness guarantees can be found in \cite{anandkumar}; we will take the TPM for granted in this paper. We summarize the TPM with whitening as presented in \cite{anandkumar} in Algorithm \ref{alg:TPM}. Note that for a $(d, d, d)$-shaped tensor $T^3$ and $(d, r)$ matrix $W$, we use $T^3(W, W, W)$ to denote the tensor contraction of the three indices of $T^3$ with the first index of $W$, i.e., 
\begin{equation}
[T^3(W, W, W)]_{mno} = \sum_{jkl}T^3_{jkl}W_{jm}W_{kn}W_{lo}.
\label{eqn:t3-contraction} 
\end{equation}

\begin{algorithm}[h]
\caption{Tensor Power Method with Whitening: $\mathtt{TPM}$}\label{alg:TPM}
\label{alg:buildtree}
\begin{algorithmic}
\REQUIRE{Symmetric three-way tensor $T^3 = \sum_{i = 1}^r \lambda_i u_i^{\otimes 3}$ with $u_i\in\bbR^d$ unit length and linearly independent and $\lambda_i>0$; rank $r$ symmetric matrix $T^2 = \sum_{i = 1}^r \lambda_i u_i^{\otimes 2}$; desired number of additive factors $r \leq d$.}
\ENSURE{Estimates $\{\hat{u}_i\}_{i=1}^r$ of $\{u_i\}_{i=1}^r$.}

\COMMENT{\texttt{WHITENING STEP}}
\STATE{Set $U$ to be the $\bbR^{d\times r}$-shaped matrix whose columns are orthonormal eigenvectors of $T_2$ corresponding to nonzero eigenvectors.}
\STATE{Set $D$ to be the $\bbR^{r\times r}$-shaped diagonal matrix of positive eigenvalues of $T_2$.}
\STATE{Compute the whitening matrix $W=UD^{-1/2}\in\bbR^{d\times r}$.}
Whiten the $T^3$ to an orthogonally decomposable tensor by setting $\tilde{T}^3 = T^3(W, W, W)\in\bbR^{r\times r\times r}$.\\
\COMMENT{\texttt{TENSOR POWER ITERATION}}
\FOR{$i = 1, \dots, r$}
    \STATE{Randomly initialize $x\in\bbR^d$.}
    \WHILE{not converged}
    \STATE{Iterate $x\gets \sum_{jk}\tilde{T}^3_{jkl}x_kx_l / \|\sum_{jk}\tilde{T}^3_{jkl}x_kx_l\|$.}
    \ENDWHILE
    \STATE{Set $\tilde{\lambda}_i = \sum_{jkl}\tilde{T}^3_{jkl}x_jx_kx_l$.}
    \STATE{Set $\tilde{u}_i = x$.}
    \STATE{Set $\hat{u}_i =\tilde{\lambda}_i(W^T)^{\dagger}\tilde{u}_i$, where $\dagger$ denotes the Moore-Penrose pseudoinverse.}
    \STATE{Deflate the whitened tensor: $\tilde{T}^3 \gets \tilde{T}^3 - \tilde{\lambda}_i\tilde{u}_i^{\otimes 3}$.}
\ENDFOR
\end{algorithmic}
\end{algorithm}

We use the TPM in our algorithm in the following way. Recall from (\ref{eqn:clean-curve-moments}) that the third moment tensor of a ground truth PWL curve $C^*$ is given by 
\begin{align}
\begin{split}
    m_3^* &= \frac{1}{12}\sum_{i=1}^M\frac{\|c^*_i - c^*_{i-1}\|}{Z}\left((c^*_{i-1}+c^*_i)^{\otimes 3} + 2{c^*_{i-1}}^{\otimes 3} + 2{c^*_i}^{\otimes 3}\right)\\
    &= \frac{1}{12}\sum_{i=1}^M\frac{\|c^*_i - c^*_{i-1}\|}{Z}\Bigg(3c_{i-1}^{\otimes 3} + 3c_i^{\otimes 3} + c_{i-1}\otimes c_i \otimes c_i + c_i\otimes c_{i-1}\otimes c_i + c_i \otimes c_i \otimes c_{i-1} \\
    &\hspace{11em} + c_i\otimes c_{i-1}\otimes c_{i-1} + c_{i-1}\otimes c_i\otimes c_{i-1} + c_{i-1}\otimes c_{i-1}\otimes c_i\Bigg).
\end{split}
\label{eqn:m3-tpm}
\end{align}

Since the coefficients on the $c_i^{\otimes 3}$ term are larger than the coefficients on the cross terms, we heuristically assume that the third moment tensor is approximately equal to a tensor of the form
\begin{equation}
\sum_{i=0}^M \lambda_i{c^*_i}^{\otimes 3}.
\end{equation}
We conjecture that discarding the cross terms in (\ref{eqn:m3-tpm}) does not affect the result of appyling the TPM too much, i.e., the TPM with whitening applied to $\widehat{m}_3$ will recover unit vectors $\{\tilde{c}_j\}_{j=0}^M$, such that after reordering, $\tilde{c}_{j_i}$ approximately spans the one-dimensional subspace that $c^*_i$ lives in, that is, $\tilde{c}_{j_i}$ is close to $c^*_i$ in cosine similarity.\footnote{The \textit{cosine similarity} of two vectors $a$ and $b$ is $\inner{a}{b}/(\|a\|\|b\|)$.} We provide numerical evidence of this claim in \S\ref{sec:tpm-cross-terms-ok} of the supplementary material. Throughout this subsection, we will use $\tilde{c}_j$ to refer to the unsorted unit vectors and $\tilde{c}_{j_i}$ to refer to the sorted unit vectors.

Note that (\ref{eqn:odeco-tensor}) is invariant to a reordering of the $\lambda_i$'s and $u_i$'s, so the fact that the TPM recovers $u_i$ up to permutation is not important to the tensor decomposition problem. However, to recover the vertices of a curve, we do in fact need to recover the ordering of the $c_i^*$'s. Therefore, we need some way to reorder the putative subspaces returned by the TPM.

\subsection{Reordering the Predicted Subspaces}
\label{sec:reordering}
The tensor power moment returns unit vectors $\{\tilde{c}_j\}_{j=0}^M$ that are ostensibly close in alignment to the ground truth points $c^*_i$; we need to reorder them so that $\tilde{c}_{j_i}$ is close to $c^*_i$, where we use the term ``close'' here (and in the rest of this subsection) in the sense of cosine similarity. The seriation problem of ordering time dependent data is well-studied and there are many existing approaches (see, for instance, \cite{yuehaw-seriation}, \cite{spectral-seriation-2}), but here we outline a heuristic approach that empirically performs well for our purposes.

A reasonable prior is that our curve is sufficiently smooth such that if $i\neq 0, M$, then the two points closest to $c^*_i$ (in cosine similarity) are $c^*_{i-1}$ and $c^*_{i+1}$. This is the inspiration behind Algorithm \ref{alg:chain-building}, where we order the subspaces (given an initial choice) by iteratively choosing the closest subspace from $\{\tilde{c}_j\}_{j=0}^M$ that has not been sorted yet.

\begin{algorithm}[h]
\caption{Chain-Building Algorithm: $\mathtt{ChainBuild}$}\label{alg:chain-building}
\begin{algorithmic}
\REQUIRE{Unordered approximate subspaces $\{\tilde{c}_{j}\}_{j=0}^M$; initial index $0\leq j_0\leq M$.}
\ENSURE{Putative ordering of the approximate subspaces $[j_0,\dots j_M]$.}
\STATE{Initialize the chain with the initial index: $\mathtt{chain} \gets [j_0]$.}
\STATE{Remove the initial index from the list of indices of unsorted subspaces: $\mathtt{RemainingSubspaces} \gets [0, \dots, M]\setminus j_0$.}
\FOR{$k = 1,\dots, M$}
    \STATE{Set $j_k$ to be the index in $\mathtt{RemainingSubspaces}$ such that $\tilde{c}_j$ is closest to $\tilde{C}_{j_{k-1}}$ in cosine similarity.}
    \STATE{Append $j_k$ to the chain: $\mathtt{chain} \gets [j_0,\dots, j_{k-1}, j_k]$.}
    \STATE{Remove $j_k$ from the list of unsorted subspaces: $\mathtt{RemainingSubspaces}\gets \mathtt{RemainingSubspaces}\setminus j_k$.}
\ENDFOR
\end{algorithmic}
\end{algorithm}

Algorithm \ref{alg:chain-building} can output up to $M+1$ different orderings depending on the choice of initial index $j_0$. To see which ordering to choose, consider the simple case illustrated in Figure \ref{fig:ordering-cartoon}. Here we assume that we have five unit vectors along some arc. Then the five possible orderings output by Algorithm \ref{alg:chain-building}, depending on the choice of initial index, are
\begin{enumerate}[(1)]
\item $c_0, c_1, c_2, c_3, c_4$
\item $c_1,c_0, c_2, c_3, c_4$
\item $c_2,c_3,c_4, c_1, c_0$
\item $c_3,c_2,c_1, c_0, c_4$
\item $c_4, c_3, c_2, c_1, c_0$.
\end{enumerate}

\begin{figure}[h]
\includegraphics[width = 0.5\textwidth]{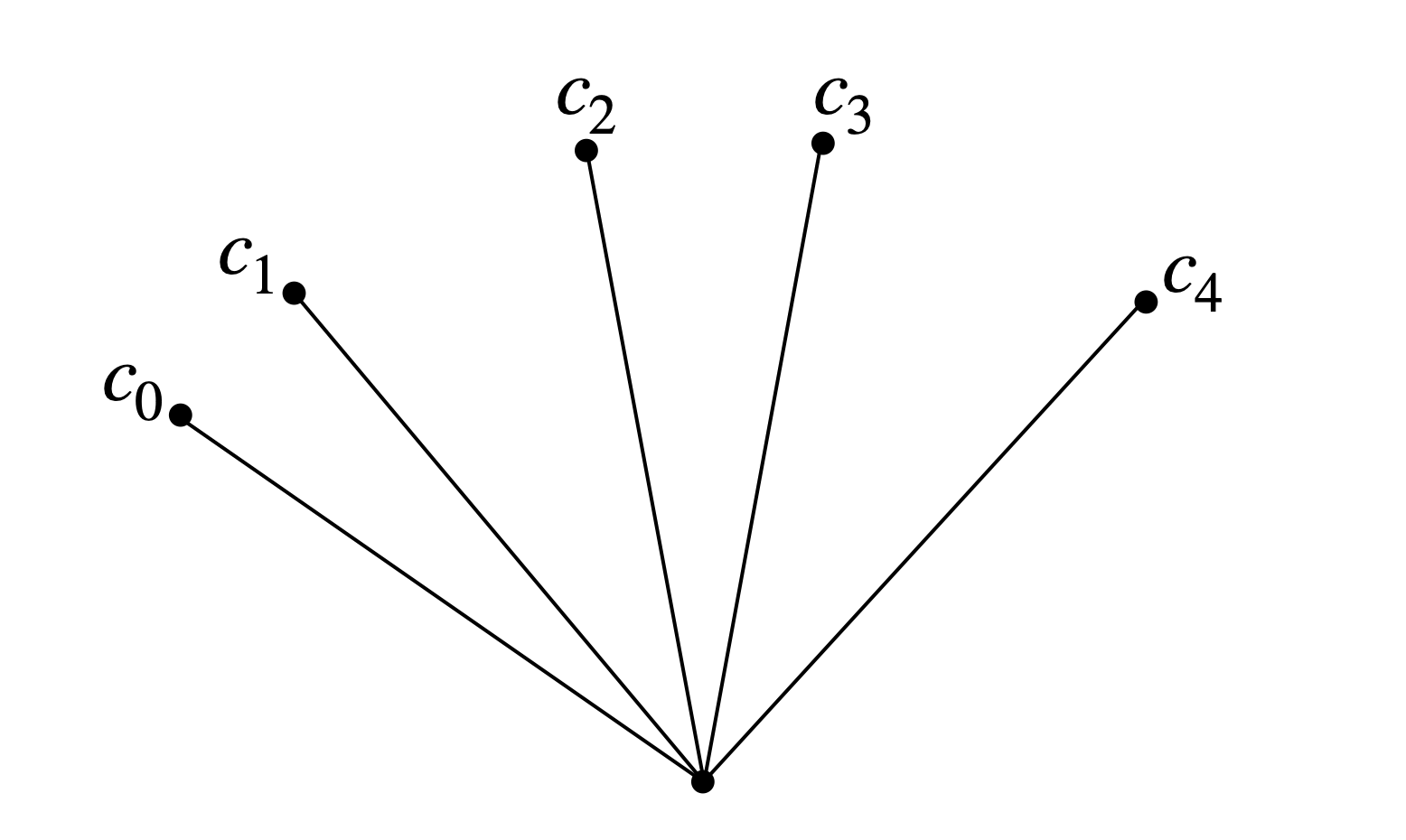}
\centering
\caption{An illustration of a toy ideal case for Algorithm \ref{alg:chain-building}.}
\label{fig:ordering-cartoon}
\end{figure}
Notice that the two correct sortings are (1) and (5), and that these two sortings are palindromes of another. A reasonable heuristic then is to assume that $C^*$ does not deviate too much from this ideal situation and $\{\tilde{c}_j\}$ do not deviate too much from the true subspaces, build up the $M+1$ chains starting from each of the predicted subspaces, and choose the chain that has a palindrome. 

Of course, in general we are not guaranteed that a palindrome pair will exist among the $M+1$ sortings generated by Algorithm \ref{alg:chain-building}, so we instead need a way to measure how ``palindromic'' a particular sorting is given a list of sortings.
\begin{defn}
Let $\Pi = \{\pi_i\}$ be a set of permutations of $M+1$ elements, where we identify permutations with ordered tuples of length $M+1$ with exactly one of $\{0,\dots M\}$ in each entry. Let $\Rev(\pi)$ be the reversed tuple of $\pi$. Given $\pi_1, \pi_2\in\Pi$, we define the palindrome similarity $\psim(\pi_1,\pi_2)$ between them as the number of entries in which $\pi_1$ and $\Rev(\pi_2)$ agree. For example, $\psim((0, 1, 2, 3, 4), (2, 3, 4, 1, 0)) = 3$. The \textit{palindromicity} $\mathscr{P}_{\Pi}(\pi)$ of a permutation $\pi\in\Pi$ is defined as the maximum palindrome similarity between $\pi$ and every other permutation in $\Pi$:
\begin{equation}
\scrP_{\Pi}(\pi) = \max_{\nu\in\Pi}\psim(\pi, \nu).
\label{eqn:scrP}
\end{equation}
\end{defn} 
Note that if the palindromicity of a permutation in a list is equal to $M+1$, that means the palindrome of the permutation exists in the list. From our list of $M+1$ orderings generated by the chain-building algorithm, we then choose the one with the highest palindromicity, breaking ties arbitrarily. In the ideal case, there should be exactly one pair of orderings with palindromicity equal to $M+1$.

\subsection{Final Optimization Steps}
\label{sec:final-optim}
From the output of Algorithm 1 sorted by Algorithm 2, we now have a row-stacked matrix $\tilde{C} = [\tilde{c}_{j_0},\dots\tilde{c}_{j_M}]^T$ of unit vectors that approximately span the subspaces in which $c_i^*$ lie. Since $\tilde{c}_{j_i}$ are unit length, we now wish to find scalars $\lambda_i$ such that the curve with vertices $\lambda_i\tilde{c}_{j_i}$ has moments close to the estimated moments $\widehat{m}_1, \widehat{m}_2, \widehat{m}_3$.  Observe that $\diag(\lambda)\tilde{C}$ is the matrix whose $i$th row is given by $\lambda_i\tilde{c}_{j_i}$. Then we choose $\lambda$ to minimize the loss of the first three moments of $\diag(\lambda)\tilde{C}$:
\begin{equation}
\hat{\lambda} = \argmin_\lambda \left[\sum_{k = 1}^3 \left\|m_k(\diag(\lambda)\tilde{C}) - \widehat{m}_k\right\|_F^2\right]
\label{eqn:m123-loss}
\end{equation}
If the $\tilde{c}_{j_i}$ are close in alignment to $c^*_i$, then $\diag(\hat{\lambda})\tilde{C} \eqdef \hat{C}_{\textrm{init}}$ should be close to the ground truth $C^*$. If this is close enough, then it will be within the neighborhood of $C^*$ where the third moment minimizer is unique (the existence of such a neighborhood is precisely the statement of Corollary \ref{cor:m3-well-posed}). We call this first step the initial estimate phase.

With $\hat{\lambda}$ determined, we now use $\hat{C}_{\textrm{init}}$ as an initialization for gradient descent on $L_{m_3^*}(C, p)$ as defined in (\ref{eqn:loss-defns}). In practice, as mentioned in \S\ref{sec:pwl-curve-moments}, since $C^*$ naturally lives in an $M$-dimensional subspace $V$ of $\bbR^d$, we first project $\hat{C}_{\textrm{init}}$ down to the subspace spanned by the top $M$ eigenvectors of $\widehat{m}_2$; let $Q$ be the $(d, M)$-shaped matrix whose columns are these eigenvectors. Note that if we replace the estimate $\widehat{m}_2$ with the exact second moment $m_2^*$ of $C^*$, then $m_2^*$ has exactly $M$ nonzero eigenvectors. Then the projection of $\hat{C}_{\textrm{init}}$ onto $V$ can be estimated by $\hat{C}_{\textrm{init}}Q$ (the estimation is exact when $Q$ is computed from $m_2^*$ rather than $\widehat{m}_2$).

Since we wish to match the third moment in the subspace $V$, we need to estimate the third moment of the projected curve. The exact third moment of the projected curve is given by $m_3^*(Q, Q, Q)$ (see (\ref{eqn:t3-contraction})), and hence we can estimate the projected third moment from data by $\widehat{m}_3(Q, Q, Q)$. We then use $\hat{C}_{\textrm{init}}Q$ and its corresponding proportional segment lengths $\hat{p}_{\textrm{init}}$ as an initial point for minimizing $L_{\widehat{m}_3(Q, Q, Q)}(C, p)$ via alternating gradient descent, then deproject the result back to $\bbR^d$ to get our final prediction $\hat{C}$. We call this second step consisting of projection, alternating descent, and deprojection the \textit{finetuning phase}.

To summarize, our recovery algorithm is in two phases. The first initial estimate phase uses the TPM on the third moment tensor to approximately recover the unordered subspaces in which the points $c_i$ approximately lie; these points are then sorted using the chain-building algorithm, and the sorting with the best palindromicity is chosen. We then find the best $\hat{c}_i$ along these subspaces that minimize the loss of the first three moments. In the second finetuning phase, we use these $\hat{c}_i^*$ as an initial guess for minimizing the loss of the relaxed third moment in the $M$-dimensional subspace where $C^*$ naturally lies. The full algorithm is presented as Algorithm \ref{alg:full-alg}.

\begin{algorithm}
\caption{Full Curve Recovery Algorithm}\label{alg:full-alg}
\begin{algorithmic}
\REQUIRE{Number of segments $M$; estimated moments $\widehat{m}_1, \widehat{m}_2, \widehat{m}_3$ of a PWL curve in $\bbR^d$ where $d\geq M$.}
\ENSURE{Vertices of predicted PWL curve $\hat{C}$ in $\bbR^d$ with $M$ segments.}
\COMMENT{\texttt{FIRST PHASE, INITIAL ESTIMATE}}
\STATE{Use the tensor power method to estimate the unordered subspaces in which the vertices of $C^*$: $\tilde{C}_{\textrm{unsorted}} \defeq \{\tilde{c}_0,\dots,\tilde{c}_M\} = \texttt{TPM}(\widehat{m}_3, \widehat{m}_2, M)$.}
\STATE{Build all possible putative sortings of the unsorted subspaces: $\Pi = \{\mathtt{ChainBuild}(\{\tilde{c}_{j}\}_{j=0}^M, i)\}_{i=0}^M$.}
\STATE{Choose the most palindromic sorting in $\Pi$: $\hat{\pi} = \argmax_{\pi}\mathscr{P}_{\Pi}(\pi)$ (see (\ref{eqn:scrP})); break ties arbitrarily.}
\STATE{Sort $\tilde{C}_{\textrm{unsorted}}$ using $\hat{\pi}$: $\tilde{C} = [\tilde{c}_{\hat{\pi}(0)},\dots, \tilde{c}_{\hat{\pi}(M)}]$.}
\STATE{Find $\hat{\lambda}\in\bbR^{M+1}$ that minimizes (\ref{eqn:m123-loss}) via gradient descent: $\hat{\lambda} = \argmin_\lambda \left[\sum_{k = 1}^3 \left\|m_k(\diag(\lambda)\tilde{C}) - \widehat{m}_k\right\|_F^2\right]$.}
\STATE{Compute the initial guess for the curve vertices along the subspaces spanned by $\{\tilde{c}_i\}$: $\hat{C}_{\textrm{init}} = \diag(\hat{\lambda})\tilde{C}$.}
\COMMENT{\texttt{SECOND PHASE, FINETUNING}}
\STATE{Set $Q$ to be the $(d, M)$-shaped matrix whose columns are the eigenvectors corresponding to the top $M$ eigenvalues of $\widehat{m}_2$.}
\STATE{Project the initial projection to the subspaces spanned by the columns of $Q$: $\hat{C}_{\proj\textrm{init}} = \hat{C}_{\textrm{init}}Q$.}
\STATE{Set $\hat{p}_{\proj\textrm{init}}$ as the proportional segment lengths of $\hat{C}_{\proj\textrm{init}}$.}
\STATE{Compute the third moment of the projected curve: $(\widehat{m}_3)_{\proj} = \widehat{m}_3(Q, Q, Q)$.}
\STATE{Solve $(\hat{C}_{\proj}, \hat{p}_{\proj}) = \argmin_{C, p} L_{(\widehat{m}_3)_{\proj}}(C, p)$ via alternating gradient descent on $C$ and $p$ with $(\hat{C}_{\proj\textrm{init}}, \hat{p}_{\proj\textrm{init}})$ as initialization.}
\STATE{Deproject the final projected prediction to the ambient space: $\hat{C} = \hat{C}_{\proj}Q^T$}
\end{algorithmic}
\end{algorithm}

\section{Numerical Results}
\label{sec:numerical-results}
Here we present some numerical results on using Algorithm \ref{alg:full-alg} to recover curves. Our implementation is primarily built with the the libraries JAX and JAXopt (\cite{jax}, \cite{jaxopt}) and available on \href{https://github.com/PhillipLo/curve-recovery-public}{GitHub}.\footnote{\url{https://github.com/PhillipLo/curve-recovery-public}}

To generate our data, we first prescribe $M$ segment lengths $Z_i$ uniformly randomly from the interval $[1, 2]$. Then we sample a random curve in $\bbR^d$ with $M$ segments by choosing the origin as the initial point, then iteratively adding points by choosing a random vector on the sphere with radius $L_i$ and adding it to the most recently generated point on the curve. We then subtract away the first moment of the curve to ensure that our curve is mean zero. We do not normalize the segment lengths to fit the curve into the unit ball (as in the hypotheses for Proposition \ref{prop:main-chi2-bd}) for numerical stability reasons.

Recall Algorithm \ref{alg:full-alg} relies on having access to the noise-free moments of the ground truth curve $C^*$. For high noise, a very large number (order $\sigma^6$) of data points from a noisy cloud are required to closely estimate the first three moments of the noise-free underlying curve. In high dimensions, this becomes enormously computationally expensive. Therefore, we show results from two different experiments: one where we apply Algorithm \ref{alg:full-alg} directly on ground truth noise-free moments that we obtain from $C^*$ (i.e., we set $\widehat{m}_k = m_k^*$), and another experiment that uses empirical moments computed from a point cloud generated from $C^*$ (i.e. we estimate $\widehat{m}_k$ from data). Due to the large number of points needed to accurately estimate the third moment tensor, we compute the moments in an online manner so that we never have to store an entire point cloud in memory. We discuss more about the number of samples needed to estimate the third moment tensor in \S\ref{sec:third-mom-estimation}.

Algorithm \ref{alg:full-alg} involves an initial phase that estimates a curve close to the ground truth before a second phase that directly minimizes the relaxed third moment loss. We propose two similar baseline algorithms to show that our initial rough estimation phase is essential: we simply directly minimize either the squared Frobenius loss of the third moment or the sum of the losses of the first, second and third moments from multiple random initializations, as detailed in Algorithm \ref{alg:baseline-alg}. We perform the optimization in an estimation of the natural $M$-dimensional subspace $V$ where the curve lies, so we minimize against the moments of the projected curve. As before, let $Q$ be the $(d, M)$-shaped matrix whose columns are the top $M$ eigenvectors of $\widehat{m}_2$; when $\widehat{m}_2$ equals $m_2^*$ exactly, the columns of $Q$ exactly span $V$. Then the first three moments of the projected curve can be estimated by $(\widehat{m}_1)_{\proj} = Q^T\widehat{m}_1$, $(\widehat{m}_2)_{\proj} = Q^T\widehat{m}_2Q$, and $(\widehat{m}_3)_{\proj} = \widehat{m}_3(Q, Q, Q)$. As in Algorithm \ref{alg:full-alg}, the minimizations are performed with gradient descent. We perform the minimization with $n = 10$ random initializations and choose the result with the best final moment loss.

\begin{algorithm}
\caption{Baseline Recovery Algorithm with Random Initialization for Moment Matching}\label{alg:baseline-alg}
\begin{algorithmic}
    
\REQUIRE{Number of segments $M$; estimated moments $\widehat{m}_1, \widehat{m}_2, \widehat{m}_3$ of a PWL curve in $\bbR^d$ where $d\geq M$, number of random baseline initializations $n$.}
\ENSURE{Vertices of predicted PWL curve $\hat{C}$ in $\bbR^d$ with $M$ segments.}
\STATE{Use eigenvalue decomposition of $\widehat{m}_2$ to compute the $(d,M)$-shaped matrix $Q$ whose columns are the top $M$ eigenectors of $\widehat{m}_2$.}
\STATE{Estimate the moments of the projected curve: $(\widehat{m}_1)_{\proj} = Q^T\widehat{m}_1$, $(\widehat{m}_2)_{\proj} = Q^T\widehat{m}_2Q$, $(\widehat{m}_3)_{\proj} = \widehat{m}_3(Q, Q, Q)$.}
\STATE{Randomly initialize the running best predicted curve $\hat{C}_{\proj\textrm{best}}$ and set the running best loss from $n$ trials: $\textrm{loss}_{\best} \gets +\infty$.}
\STATE{Define a loss function from matching just the third moment or the all of first three moments: $L(C)\defeq \sum_{k = 1}^3 \left\|m_k(C) - \widehat{m}_k\right\|_F^2$ or $L(C)\defeq \left\|m_3(C) - \widehat{m}_3\right\|_F^2$ (as in (\ref{eqn:loss-defns})).}
\COMMENT{\texttt{iteratively update best estimate}}
\FOR{$k = 1,\dots, n$}
    \STATE{Randomly initialize $\hat{C}_{\proj\textrm{init}}$ with shape $(M+1, d)$.}
    \STATE{Solve $\hat{C}_{\proj} = \argmin_{C}L(C)$ via gradient descent with $\hat{C}_{\proj\textrm{init}}$ as initialization.}
    \IF{$L(\hat{C}_{\proj}) < \emph{loss}_{\best}$}
    \STATE{Update the best loss: $\textrm{loss}_{\best} \gets L(\hat{C}_{\proj})$.}
    \STATE{Update the best projected curve: $\hat{C}_{\proj\textrm{best}}\gets\hat{C}_{\proj}$.}
    \ENDIF
\ENDFOR

Deproject the best projected curve: $\hat{C} = \hat{C}_{\proj\textrm{best}}Q^T$.
\end{algorithmic}
\end{algorithm}

In Table \ref{tab:cloud-cnst}, we present results from the experiment with moments $\widehat{m}_k$ estimated from a point cloud with $N=10^8$ points, with $M=16$ segments, ambient dimension $d = 24$, and noise level $\sigma^2 = 1/4$. We report the lower quartile, median, and upper quartile of the square root of the curve loss $\rho(C^*,\hat{C})$ (see (\ref{eqn:curve-dist})) and relative third moment loss $\|m_3^* - m_3(\hat{C})\|_F / \|m_3^*\|_F$ between the ground truth and predicted curves over 500 trials with different random ground truth curves $C$. We present the results for the outputs of both phases of Algorithm \ref{alg:full-alg}, as well as the performance of both baselines with the best result out of $n = 10$ random initializations. In Table \ref{tab:true-cnst}, we present the same statistics for the experiment with access to ground truth moments with $M=32$, and $d = 48$ (we are able to perform the experiment in higher dimensions because we no longer have to generate and compute moments of a large point cloud, which is the primary performance bottleneck). We see that in both experiments, Algorithm \ref{alg:full-alg} significantly outperforms both baselines in the curve loss, even without the finetuning phase. Observe also that the actual third moment loss is lower for the baseline algorithms even though the curve loss is higher. This indicates that the baseline algorithms are converging to a curve different from the ground truth $C^*$; this can be seen in Figures \ref{fig:cloud-cnst} and \ref{fig:true-cnst}, where the result of the baseline algorithm is a curve clearly much different from the ground truth. Indeed, this suggests that while determination of a curve from the third moment is \textit{locally} well-posed, it is not \textit{globally} well-posed, and indeed there can be many curves with very similar, if not identical, third moment. 

In the experiment with access to ground truth moments, the sorting of subspaces was successful $485/500 = 97.2\%$ times, where we define success as there being exactly one palindromic pair of sorted subspaces. In the experiment with moments estimated from a point cloud, the success rate was $371/500 = 74.2\%$. We show some predicted curves from Algorithm \ref{alg:full-alg} for both experiments in Figures \ref{fig:cloud-cnst} and \ref{fig:true-cnst}.
\begin{table}
    \small
    \caption{Lower quartile, median, and upper quartile root curve losses $\sqrt{\rho(C^*,\hat{C})}$ (see (\ref{eqn:curve-dist})) and relative third moment losses $\|m_3^* - m_3(\hat{C})\|_F / \|m_3^*\|_F $ over 500 trials with moments estimated from point cloud with $N = 10^8$ points, $M=16$ segments, ambient dimension $d = 24$, noise level $\sigma^2=1/4$. The lowest value in each column is bolded.} \label{tab:cloud-cnst} 
    \centering
    \begin{tabular}{lcccccc}
    \toprule
    \multicolumn{1}{l}{} & \multicolumn{3}{c}{\textbf{root curve loss}} & \multicolumn{3}{c}{\textbf{rel. third moment loss}} \tabularnewline
    \cmidrule(lr){2-4} \cmidrule(lr){5-7}
    \multicolumn{1}{l}{\textbf{method}} & \textbf{25\%} & \textbf{50\%} & \textbf{75\%} & \textbf{25\%} & \textbf{50\%} & \textbf{75\%} \tabularnewline
    \midrule
    \makecell[l]{Algorithm \ref{alg:full-alg}\\ after first phase} & 0.4766 & 0.5263 & 0.7602 & 0.0577 & 0.0682 & 0.0892
    \tabularnewline
    \makecell[l]{Algorithm \ref{alg:full-alg}\\ after second phase} & \bftab 0.2479 & \bftab 0.3027 & \bftab 0.7091
    & 0.0663 & 0.1009 & 0.1570 \tabularnewline
    \midrule
    \makecell[l]{Baseline algorithm \\with third moment only} & 1.6202 & 1.6910 & 1.7544 & 0.0171 & 0.0203 & 0.0246 \tabularnewline
    \makecell[l]{Baseline algorithm \\with all three moments} &1.5975 & 1.6647 & 1.7337 & \bftab 0.0139 & \bftab 0.0157 & \bftab 0.0171 \tabularnewline
    \bottomrule
    \end{tabular}
    \vspace{3em}
    \caption{Lower quartile, median, and upper quartile root curve losses $\sqrt{\rho(C^*,\hat{C})}$ (see (\ref{eqn:curve-dist})) and relative third moment losses $\|m_3^* - m_3(\hat{C})\|_F / \|m_3^*\|_F $ over 500 trials  with access to ground truth moments, $M=32$ segments, ambient dimension $d = 48$. The lowest value in each column is bolded.}\label{tab:true-cnst}
    
    \centering
    \begin{tabular}{lcccccc}  
    \toprule
    \multicolumn{1}{l}{} & \multicolumn{3}{c}{\textbf{root curve loss}} & \multicolumn{3}{c}{\textbf{rel. third moment loss}} \tabularnewline
    \cmidrule(lr){2-4} \cmidrule(lr){5-7}
    \multicolumn{1}{l}{\textbf{method}} & \textbf{25\%} & \textbf{50\%} & \textbf{75\%} & \textbf{25\%} & \textbf{50\%} & \textbf{75\%} \tabularnewline
    \midrule
    \makecell[l]{Algorithm \ref{alg:full-alg}\\ after first phase} & 0.6823 & 0.7998 & \bftab 0.9911 & 0.0428 & 0.0636 & 0.1056\tabularnewline
    \makecell[l]{Algorithm \ref{alg:full-alg}\\ after second phase} & \bftab 0.4438 & \bftab 0.6225 & 1.0536 & 0.0698 & 0.1045 & 0.2133\tabularnewline
    \midrule
    \makecell[l]{Baseline algorithm \\with third moment only} & 1.9966 & 2.0521 & 2.1094 & 0.0153 & 0.0174 & 0.0204\tabularnewline
    \makecell[l]{Baseline algorithm \\with all three moments} & 1.9905 & 2.0435 & 2.1054 & \bftab 0.0118 & \bftab 0.0127 & \bftab 0.0137\tabularnewline
    \bottomrule 
    \end{tabular}
\end{table}

\begin{figure}
\includegraphics[width = \textwidth]{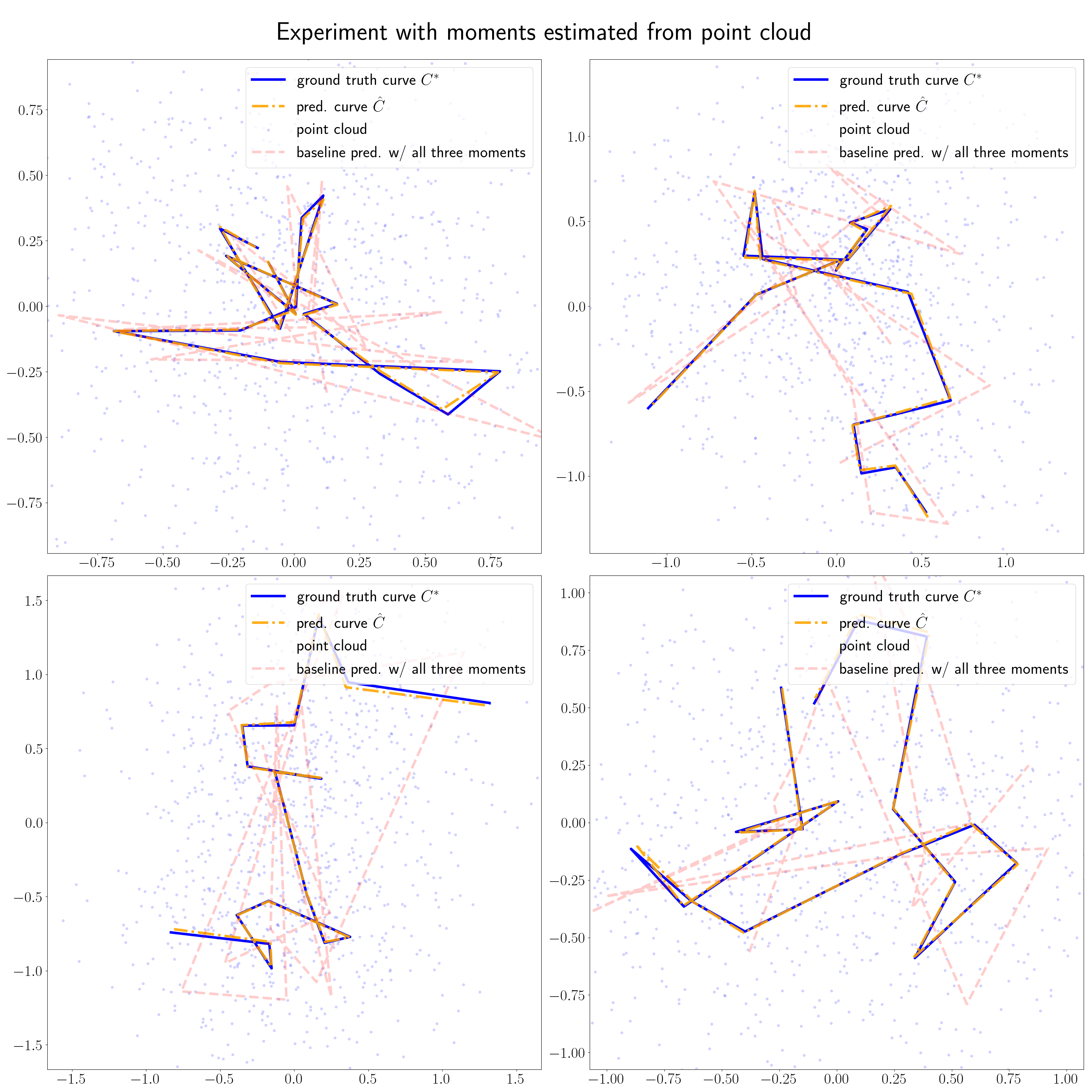}
\caption{Four uncurated experimental results for curves with $M=16$ segments, ambient dimension $d = 24$, noise level $\sigma^2=1/4$ and moments estimated from point cloud with $N = 10^8$ points. We show the ground truth curve and a subsample of 1000 points from the point cloud in blue. The output of Algorithm \ref{alg:full-alg} is given by the dashed orange line. The output of the baseline algorithm with matching to all three moments is shown faintly in red. All curves are projected down to the first two dimensions of $\bbR^{24}$}
\label{fig:cloud-cnst}
\end{figure}

\begin{figure}
\includegraphics[width = \textwidth]{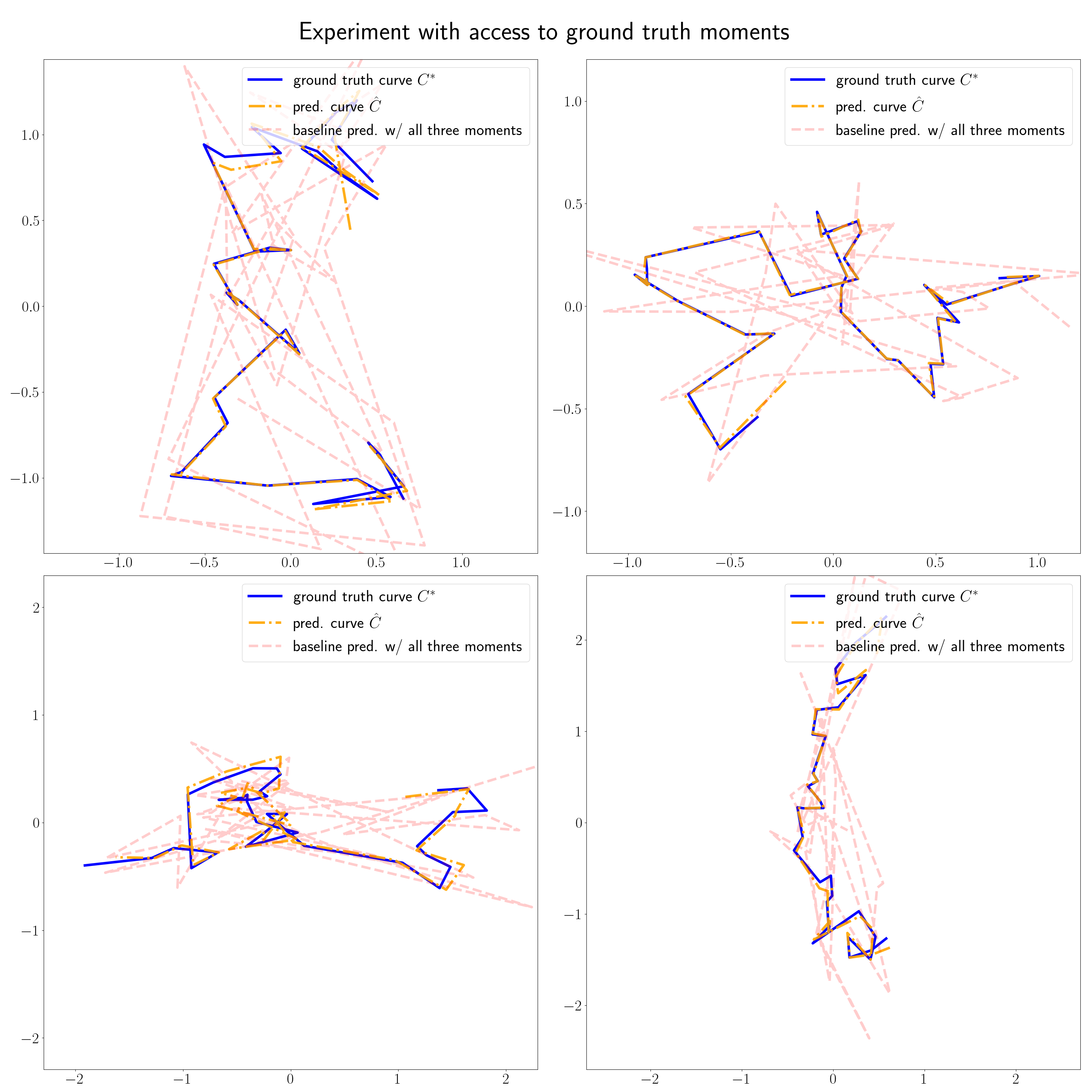} 
\caption{Four uncurated experimental results for curves with $M=32$ segments, ambient dimension $d = 48$, and access to ground truth moments. We show the ground truth curve in blue. The output of Algorithm \ref{alg:full-alg} is given by the dashed orange line. The output of the baseline algorithm with matching to all three moments is shown faintly in red. All curves are projected down to the first two dimensions of $\bbR^{48}$}
\label{fig:true-cnst}
\end{figure}

\section{Conclusion and Future Work}

In this paper, we consider the problem of resolving a PWL constant speed curve in high dimensions from a high noise point cloud around the curve. We show that for sufficiently large noise, any algorithm hoping to recover the underlying curve with high probability will require at least $O(\sigma^6)$ samples, giving an asymptotic lower bound on the sample complexity. We then show that this lower bound is asymptotically tight by providing a recovery algorithm based on the third moment tensor, exploiting the fact that the third moment uniquely determines a curve in some neighborhood. Further work is needed in showing theoretical guarantees for our algorithm, e.g., showing robustness in the presence of error in the estimation of the curve moments $m_k(C)$ from the data, or guarantees on Algorithm \ref{alg:chain-building}. There is also work to be done in recovery of curves other than PWL curves, as well as investigating the phase transition between the low noise and high noise recovery problems.

\section*{Acknowledgments}
PL would like to thank Christopher Stith for enlightening discussions on the differential geometry arguments used in the proof of Proposition \ref{prop:matching-m1m2}. YK would like to thank Xin Tong and Wanjie Wang from the National University of Singapore for discussions regarding the potential usage of moments in recovering a curve. We also thank the editor and anonymous reviewers at SIMODS for carefully reading the manuscript and offering helpful suggestions.

YK was supported partially by DMS-2111563, DMS-2339439, DE-SC0022232, and the Sloan Foundation. PL was supported partially by DE-SC0022232.

\appendix 

\section{Cloud Tracing Algorithm for the Low-Noise Regime}
\label{sec:cloud-tracing}
In this section, we very briefly outline our algorithm for tracing through low-noise point clouds. We do not perform any large-scale experiments or detailed analysis of the method. We emphasize that the purpose of this section is primarily pedagogical and meant to demonstrate that the curve recovery problem is qualitatively different in the low-noise regime.

For simplicity, we consider point clouds coming from a PWL curve with an \textit{a priori} known constant segment lengths $L$.  First consider the issue of resolving an ``elbow'' with three points, $c_0, c_1, c_2\in\bbR^d$, where $c_1$ is known. Since noise is low, we can compute the point cloud local to this elbow by considering all points in the cloud within distance $L$ of $c_1$. The two-dimensional subspace formed by the three points can be estimated from the second empirical moment, onto which we can project the local point cloud. The direction (which we can identify with an angle in the two-dimensional subspace) of each vector pointing from $c_1$ to a projected point from the local cloud can be computed. The histogram of angle counts should be bimodal, with the peaks $\theta_0$ and $\theta_2$ occurring at the directions in which the projections of $c_0$ and $c_2$ lie. Since the segment lengths are known, we can then estimate the projections of $c_0$ and $c_2$ and deproject back to $\bbR^d$.

Given a full curve and an estimate of $c_0$, we can first compute the point cloud local to $c_0$, that is, all points that are within distance $L$ away from $c_0$. We can apply a similar technique as above on this local cloud and $c_0$ to get an estimate for $c_1$, then iteratively apply the elbow resolution technique to the elbow centered at $c_i$ and the point cloud local to $c_i$ to estimate $c_{i+1}$. To account for the error that might propagate at each step, we can average the two curves obtained by applying this method with $c_0$ and $c_M$ as the starting points. Figure \ref{fig:tracing-lohi} shows the outcome of this method on curves in $\bbR^2$ applied to only the initial point $c_0$.

The reason why this method fails for the high noise case is that the local point clouds are no longer meaningful. When noise is low, the points within distance $L$ of $c_i$ have a high likelihood of coming from the elbow given by $c_{i-1}$, $c_i$, and $c_{i+1}$, and so the geometric structure of the curve around $c_i$ is preserved by the local cloud. When noise is sufficiently high, the points within distance $L$ of $c_i$ are likely to have come from elsewhere in the curve. Without performing further detailed analysis, note that since this algorithm only requires accurate estimation of the second moment, the number of points needed for this method to work is order $O(\sigma^4)$, rather than the requirement of $O(\sigma^6)$ to accurately estimate the third moment for our high noise recovery algorithm.

\section{Proof of Proposition {\ref{prop:matching-m1m2}}}
\label{sec:matching-m1m2-proof}
Let $m_1$ and $m_2$ be the first two moments of $C^*$, let $p^*\in\bbR^{M}$ be the proportional segment lengths of $C^*$, and let $Z^*>0$ be the total length of $C^*$. Then $C^*$, $p^*$, $Z^*$ is one particular solution to the following system of polynomial equations in $C = [c_0,\dots,c_M]$, $p$, $Z$ that we can write down using the relaxed moment expressions in (\ref{eqn:decoupled-moments}):
\begin{align}
  \label{eqn:system}
  \begin{split}
    \frac{1}{2}C^TA_sp &= m_1\\
    C^T\ol{A}(p)C &= m_2\\
    Z^2p_i^2 &= \|c_i - c_{i-1}\|^2, \hspace{2em} i = 1,\dots, M\\
    \sum_{i=1}^M p_i &= 1. 
  \end{split}
\end{align} 
The first line in (\ref{eqn:system}) is matching the relaxed first moment from $C$ and $p$ to the prescribed first moment $m_1$, and similarly for the second line and the second moment $m_2$. The third line ``unrelaxes'' the relaxed moments by requiring that the $p_i$ are equal to the proportional segment lengths, and the fourth line ensures that $Z$ is equal to the length of the full curve.

The first line of (\ref{eqn:system}) consists of $M$ equations, the second line consists of $M(M+1)/2$ equations ($m_2$ is a symmetric matrix, so we can eliminate all equations coming from the lower triangular part of $m_2$), the third line consists of $M$ equations, and the last line is a single equation, so we have a total of $M(M+1)/2 + 2M + 1$ equations. To count the number of variables, we have $M(M+1)$ variables from $C$, $M$ variables from $p$, and one variable from $Z$ for a total of $(M+1)^2$ variables. We therefore have more variables than equations, and hence this is an underdetermined system of polynomial equations.

If we subtract the right hand side from both sides of each line of (\ref{eqn:system}), we can consider the left hand side of the entire system as a function $f:\bbR^{(M+1)^2}\to \bbR^{M(M+1)/2 + 2M + 1}$ of $(C, p, Z)$ (where scalars, vectors, and matrices are flattened and concatenated appropriately), and we can think of the solution set to (\ref{eqn:system}) as the preimage of $0$ under $f$. 

Let $Df\in\bbR^{(M(M+1)/2 + 2M + 1)\times (M+1)^2}$ be the Jacobian of $f$ in the variables $C$, $p$, and $Z$; this is a short and wide matrix. By Claim \ref{claim:Df-rank-deficient} (see below), $Df(C^*, p^*, Z^*)$ is generically full rank. Since rank is lower semicontinuous, for generic $C^*$ there exists a neighborhood $U$ of $(C^*, p^*, Z^*)$ where $Df$ is full rank and hence surjective at all $(C, p, Z)$ in $U$. Let $f|_U:U\to\bbR^{(M+1)^2}$ be the restriction of $f$ to $U$. Then $0$ is a regular value of $f|_U$, i.e., a point such that $Df|_U$ is surjective at every point in $f|_U^{-1}(0)$. Note that we know that $f|_U^{-1}(0)$ is nonempty, since it contains at least the point $(C^*, p^*, Z^*)$. Then by the regular level set theorem from differential topology (also called the preimage theorem, see Theorem 9.9 in \cite{Tu}), $f^{-1}|_U(0)$ is a submanifold of $U$ with dimension $(M+1)^2 - [M(M+1)/2 + 2M + 1] = M(M-1)/2 > 0$. In other words, $(C^*, p^*, Z^*)$ is not an isolated point in the real algebraic variety of (\ref{eqn:system}). Therefore, for sufficiently small $\ep>0$, we can find some $(\Gamma, \pi, \zeta)\in U$ such that $\|C^*-\Gamma\|_F=\ep$ and $(\Gamma, \pi, \zeta)$ also solves (\ref{eqn:system}), and hence $\Gamma$ has the same first and second moments as $C^*$.

\begin{claim}
    \label{claim:Df-rank-deficient}
    Let $(C, p, Z)$ be a solution to (\ref{eqn:system}); then $Df(C, p, Z)$ is generically full rank. In particular, a necessary condition for $Df(C, p, Z)$ to be rank-deficient is that the all-ones vector $\mathbf{1}\in\bbR^{M+1}$ must be in the image of $C$, where here we are thinking of $C$ as a matrix.
\end{claim}
\begin{proof}
Let $f$ be as in the proof of Proposition \ref{prop:matching-m1m2} above. Then the Jacobian $Df$ has the following block form; the values in the margins of the matrix in the first line denote dimensions:
\begin{align}
    \begin{split}
        Df &= \begin{array}{ccc@{\hskip 1em} l}
            \hspace{-8.5em}(M{+}1)M & \hspace{-22em}M & \hspace{-11.5em}1 & \\[6pt]  
            \left[
            \begin{array}{ccc}
              \partial f_1/\partial C & \partial f_1/\partial p & \partial f_1/\partial Z \\
              \partial f_2/\partial C & \partial f_2/\partial p & \partial f_2/\partial Z \\
              \partial f_3/\partial C & \partial f_3/\partial p & \partial f_3/\partial Z \\
              \partial f_4/\partial C & \partial f_4/\partial p & \partial f_4/\partial Z
            \end{array}
            \right]
            &
            \begin{array}{l}
              M \\
              M(M+1)/2 \\
              M \\
              1
            \end{array}
          \end{array}\\
          &= \hspace{0.5em}\begin{bmatrix}
        \partial f_1/\partial C & \partial f_1/ \partial p & \mathbf{0} \\
        \partial f_2/\partial C & \partial f_2/ \partial p & \mathbf{0} \\
        \partial f_3/\partial C & \partial f_3/ \partial p & \partial f_3 / \partial Z \\
        \mathbf{0} & \mathbf{1} & \mathbf{0}
        \end{bmatrix}
    \end{split}
\end{align}
In the second matrix, $\mathbf{1}$ and $\mathbf{0}$ denote all-ones and all-zeros matrices of the appropriate size. We first claim that the $\partial f_1/\partial C$ block is full rank. Let $\left[\partial f_1/\partial C\right]_i$ denote the $i$th row of $\partial f_1/\partial C$, reshaped to a $(M+1, M)$-shaped matrix; these are the rows of this block. Then $\left[\partial f_1/\partial C\right]_i$ is the matrix of all zeros, except the $i$th column is equal to $A_sp/2$; it follows that the $\partial f_1 / \partial C$ block is full rank. 

We now claim that the block $\partial f_2/\partial C$ is also full rank. Let $\left[\partial f_2/\partial C\right]_{j, k}$ denote the derivative of the $j, k$ component of $f_2$ with respect to $C$, reshaped to a $(M+1, M)$-shaped matrix, with $j\geq k$ due to the symmetry of $m_2$; these are again the rows of the block. Then for $j\neq k$, the $j$th column of $\left[\partial f_2/\partial C\right]_{j, k}$ is the $k$th column of $\ol{A}(p)C$ and the $k$th column of $\left[\partial f_2/\partial C\right]_{j, k}$ is the $j$th column of $\ol{A}(p)C$; all other entries are zero. When $j= k$, the $j$th column of  $\left[\partial f_2/\partial C\right]_{j, k}$ is equal to two times the $j$th column of $\ol{A}(p)(C)$. Therefore, the only way that the rows of $\partial f_2/\partial C$ are linearly dependent is if the columns of $\ol{A}(p)C$ are dependent; but $\ol{A}(p)$ and $C$ are generically full rank, and hence the block $\partial f_2/\partial C$ is also full rank, as desired.

The block $\partial f_3/\partial Z$ is equal to the vector $2Zp^{\odot 2}$, where $\cdot^{\odot 2}$ denotes the entrywise square. Therefore, a necessary condition for $Df$ to be rank-deficient is that $A_sp$ must be in the image of the columns of $\ol{A}(p)C$, i.e., there exists a vector $x$ such that $\ol{A}(p)Cx = A_sp$. Multiplying both sides by the inverse of $\ol{A}(p)$, we get $Cx = (\ol{A}(p))^{-1}A_sp$; the vector on the right hand side is equal to the all-twos vector, and hence a necessary condition for $Df$ to be rank-deficient is that the all-ones vector must be in the image of $C$. Since $C$ is shape $(M+1, M)$ (i.e., a tall matrix), its image must be rank-deficient, and hence $Df$ is generically full rank.
\end{proof}

\section{Proof of Proposition \ref{prop:main-chi2-bd}}
\label{sec:main-chi2-bd-proof}
Our proof is modeled on the proof of Theorem A.1 in \cite{mra}. The idea of the proof is to bound the $\chi^2$ divergence with a Taylor series of the differences of moments. We first need to prove a few technical lemmas and introduce some notation. Let $\phi_\sigma(x)$ denote the density of a Gaussian centered at the origin with covariance $\sigma^2I_d$:
\begin{equation}
  \phi_\sigma(x) = \frac{1}{\sqrt{(2\pi\sigma^2)^d}}\exp\left(-\frac{\|x\|^2}{2\sigma^2}\right).
\end{equation}
Let $p_{C,\sigma}(x)$ denote the density of the noisy curve $P_{C,\sigma}$. Note that we have the expression
\begin{equation}
  \label{eqn:nc-density}
  p_{C,\sigma}(x) = \int_0^1 \phi_\sigma(x - C(t))\,dt = \int_0^1\frac{1}{\sqrt{(2\pi\sigma^2)^d}}\exp{\Bigg(}-\frac{\|x - C(t)\|^2}{2\sigma^2}{\Bigg)}\,dt.
\end{equation}
We will be using the facts that $C$ is mean zero and $\|C(t)\|<1$ throughout this section.

\begin{lem}
  \label{lem:density-lb}
  We have the lower bound 
  \begin{equation}
    \label{eqn:density-lb}
    p_{C,\sigma}(x) \geq \phi_\sigma(x)\exp\left(-\frac{1}{2\sigma^2}\right).
  \end{equation}
\end{lem}
\begin{proof}
  Using the fact that the exponential function is convex, we compute
  \begin{align*}
    p_{C,\sigma}(x) &= \int_0^1\frac{1}{\sqrt{(2\pi\sigma^2)^d}}\exp\left(-\frac{\|x - C(t)\|^2}{2\sigma^2}\right)\,dt\\
    &= \int_0^1 \frac{1}{\sqrt{(2\pi\sigma^2)^d}}\exp\left(-\frac{\|x\|^2}{2\sigma^2}\right)\exp\left(\frac{\inner{C(t)}{x}}{\sigma^2}\right)\exp\left(-\frac{\|C(t)\|^2}{2\sigma^2}\right)\,dt \\
    &\geq \phi_\sigma(x)\exp\left(-\frac{1}{2\sigma^2}\right)\int_0^1 \exp\left(\frac{\inner{C(t)}{x}}{\sigma^2}\right)\,dt\\
    &\geq \phi_\sigma(x)\exp\left(-\frac{1}{2\sigma^2}\right)\exp\left(\int_0^1\frac{\inner{C(t)}{x}}{\sigma^2}\right)\,dt\\
    &= \phi_\sigma(x)\exp\left(-\frac{1}{2\sigma^2}\right)
  \end{align*}
  The penultimate line is by Jensen's inequality, and the last line is by the fact that $C$ is mean zero.
\end{proof}

\begin{lem}
  Let $C$ and $\Gamma$ be any two curves. Then
  \label{lem:chi2-equality}
  \begin{equation}
    \label{eqn:chi2-equality}
    \int_{\bbR^d}\frac{p_{C,\sigma}(x)p_{\Gamma,\sigma}(x)}{\phi_\sigma(x)}\,dx = \int_{[0, 1]^2} \exp\left(\frac{\inner{C(t)}{\Gamma(s)}}{\sigma^2}\right)\,dtds.
  \end{equation}
\end{lem}
\begin{proof}
    Note that we have 
    \[p_{C,\sigma}(x) = \phi_\sigma(x)\int_0^1\exp\left(\frac{\inner{C(t)}{x}}{\sigma^2}\right)\exp\left(-\frac{\|C(t)\|^2}{2\sigma^2}\right)\,dt,\]
    and similarly for $p_{\Gamma,\sigma}$. Then
    {\small\begin{align*}
      \int_{\bbR^d}&\frac{p_{C,\sigma}(x)p_{\Gamma,\sigma}(x)}{\phi_\sigma(x)}\,dx \\
      &= \int_{\bbR^d}\phi_\sigma(x)\left[\int_0^1\exp\left(\frac{\inner{C(t)}{x}}{\sigma^2}\right)\exp\left(-\frac{\|C(t)\|^2}{2\sigma^2}\right)\,dt\right]\\
      &\hspace{10em}\left[\int_0^1\exp\left(\frac{\inner{\Gamma(s)}{x}}{\sigma^2}\right)\exp\left(-\frac{\|\Gamma(s)\|^2}{2\sigma^2}\right)\,ds\right]\,dx\\
      &=\int_{\bbR^d}\frac{1}{\sqrt{(2\pi\sigma^2)^d}}\exp\left(-\frac{\inner{x}{x}}{2\sigma^2}\right)\Bigg[\int_{[0, 1]^2}\exp\left(\frac{2\inner{C(t) + \Gamma(s)}{x}}{2\sigma^2}\right)\\
      &\hspace{10em}\exp\left(-\frac{\|C(t)\|^2 + \|\Gamma(s)\|^2}{2\sigma^2}\right)\,dtds\Bigg]dx\\
      &=\int_{[0, 1]^2}\Bigg[\int_{\bbR^d}\frac{1}{\sqrt{(2\pi\sigma^2)^d}}\exp\left(-\frac{\inner{x}{x} - 2\inner{C(t) + \Gamma(s)}{x} + \inner{C(t) + \Gamma(s)}{C(t) + \Gamma(s)}}{2\sigma^2}\right)\\
      &\hspace{5em}\exp\left(\frac{\inner{C(t) + \Gamma(s)}{C(t) + \Gamma(s)}}{2\sigma^2}\right)\exp\left(-\frac{\|C(t)\|^2 + \|\Gamma(s)\|^2}{2\sigma^2}\right)\,dx\Bigg]dtds\\
      &=\int_{[0, 1]^2}\underbrace{\int_{\bbR^d}\frac{1}{\sqrt{(2\pi\sigma^2)^d}}\exp\left(-\frac{\|x - (C(t) + \Gamma(s))\|^2}{2\sigma^2}\right)\,dx}_{= 1}\\
      &\hspace{3em}\exp\left(\frac{\|C(t)\|^2 + \|\Gamma(s)\|^2 + 2\inner{C(t)}{\Gamma(s)} - \|C(t)\|^2 - \|\Gamma(s)\|^2}{2\sigma^2}\right)\,dtds\\
      &= \int_{[0, 1]^2}\exp\left(\frac{\inner{C(t)}{\Gamma(s)}}{\sigma^2}\right)\,dtds.\\
    \end{align*}}
    In line 4, we are using Fubini's theorem and completing the square.
\end{proof}

\begin{lem}
  \label{lem:E-diff-bd}
  Let $C(t)$, $\Gamma(t)$ be parametric curves such that $\|C(t) - \Gamma(t)\|\leq 1/3$ and $\|C(t)\|,\|\Gamma(t)\|\leq 1$ for all $t\in[0,1]$. With $\rho(C,\Gamma)$ as defined in (\ref{eqn:curve-dist}), we have the upper bound 
  \begin{equation}
  \label{eqn:E-diff-bd}
  \left\|m_k(C) - m_k(\Gamma)\right\|_F^2 \leq 12\cdot2^k\rho(C,\Gamma).
  \end{equation}

\end{lem}
\begin{proof}
  Using Jensen's inequality, we compute
  \begin{align*}
    \left\|m_k(C) - m_k(\Gamma)\right\|_F^2 &= \left\|\int_0^1 C(t)^{\otimes k} - \Gamma(t)^{\otimes k} \,dt\right\|_F^2\\
    &\leq \int_0^1 \left\|C(t)^{\otimes k} - \Gamma(t)^{\otimes k}\right\|_F^2\,dt\\
    &\leq 12\cdot 2^k \int_0^1 \left\|C(t) - \Gamma(t)\right\|^2\,dt\tag{Lemma \ref{lem:tensor-pwr-diff-bound}}\\
    &=12\cdot 2^k \rho(C,\Gamma).
  \end{align*}
\end{proof}

We are now ready to prove Proposition \ref{prop:main-chi2-bd}.
\begin{proof}[Proof (of Proposition \ref{prop:main-chi2-bd})]
  Writing out the definition of the $\chi^2$ divergence and using the fact that $\sigma\geq 1$, we have 
  {\small\begin{align*}
    \chi^2(P_{C,\sigma}\| P_{\Gamma,\sigma}) &= \int_{\bbR^d}\frac{(p_{C,\sigma}(x) - p_{\Gamma,\sigma}(x))^2}{p_{C,\sigma}(x)}\,dx\\
    &\leq \underbrace{\exp\left(\frac{1}{2\sigma^2}\right)}_{\leq e^{1/2}}\int_{\bbR^d}\frac{(p_{C,\sigma}(x) - p_{\Gamma,\sigma}(x))^2}{\phi_\sigma(x)}\,dx\tag{Lem. \ref{lem:density-lb}}\\
    &\lesssim\int_{[0,1]^2}\exp\left(\frac{\inner{C(t)}{C(s)}}{\sigma^2}\right) - 2\exp\left(\frac{\inner{C(t)}{\Gamma(s)}}{\sigma^2}\right) + \exp\left(\frac{\inner{\Gamma(t)}{\Gamma(s)}}{\sigma^2}\right)\,dtds\tag{Lem. \ref{lem:chi2-equality}}\\
    &\leq\int_{[0,1]^2}\sum_{k\geq 0}\frac{1}{\sigma^{2k}k!}\left(\inner{C(t)}{C(s)}^k - 2\inner{C(t)}{\Gamma(s)}^k + \inner{\Gamma(t)}{\Gamma(s)}^k\right)\,dtds\tag{Taylor exp.}\\
    &=\sum_{k\geq 0}\frac{1}{\sigma^{2k}k!}\int_{[0,1]^2}\Bigg[\inner{C(t)^{\otimes k}}{C(s)^{\otimes k}} \tag{Fubini}\\
    &\hspace{10em}- 2\inner{C(t)^{\otimes k}}{\Gamma(s)^{\otimes k}} + \inner{\Gamma(t)^{\otimes k}}{\Gamma(s)^{\otimes k}}\Bigg]\,dtds\\
    &=\sum_{k\geq 0}\frac{1}{\sigma^{2k}k!}\Bigg[\inner{m_k(C)}{m_k(C)} - 2\inner{m_k(C)}{m_k(\Gamma)} + \inner{m_k(\Gamma)}{m_k(\Gamma)}\Bigg]\\
    &=\sum_{k\geq 0}\frac{1}{\sigma^{2k}k!}\|m_k(C)-m_k(\Gamma)\|^2_F\\
    &\leq \sum_{k\geq \ell}\frac{12\cdot 2^k}{\sigma^{2k}k!}\rho(C,\Gamma)\tag{Lem. \ref{lem:E-diff-bd}}\\
    &= \calK_\ell\sigma^{-2\ell}\rho(C,\Gamma).
  \end{align*}}
  In line 3, $\lesssim$ denotes inequality up to a universal constant (in our case, $e^{1/2}$). In line 5, we are using the fact that $\inner{x}{y}^k = \inner{x^{\otimes k}}{y^{\otimes k}}$, where $\inner{\cdot}{\cdot}$ denotes the Frobenius inner product for tensor arguments. In the penultimate line we are using the fact that the moments match up to order $\ell-1$.
\end{proof}

\section{Vectorized Notation for Moments}
\label{sec:custom-notation}
Here we elaborate on the notation in Definition \ref{def:decoupled-curve-moments}. We use $A^s$, $A^l$, and $A^r$ to denote the following three $(M+1, M)$-shaped matrix constants (here $M=3$ for illustrative purposes):
\begin{equation}
A^s = \begin{bmatrix}1&0&0\\1&1&0\\0&1&1\\0&0&1\end{bmatrix} \; A^l = \begin{bmatrix}1&0&0\\0&1&0\\0&0&1\\0&0&0\end{bmatrix}\; A^r = \begin{bmatrix}0&0&0\\1&0&0\\0&1&0\\0&0&1\end{bmatrix}\
\label{eqn:A-matrices}
\end{equation}
($s$ stands for sum, $l$ for left, and $r$ for right, since the rows of $(A^s)^TC$ are $c_{i-1} + c_i$, the rows of $(A^l)^TC$ are $c_{i-1}$, and the rows of $(A^r)^TC$ are $c_i$). The $(M+1,M+1)$-shaped matrix $\ol{A}(p)$ is given by
\begin{equation}
  \ol{A}(p) \defeq \frac{1}{6}\left(A^s\diag(p)(A^s)^T + A^l\diag(p)(A^l)^T + A^r\diag(p)(A^r)^T\right),
  \label{eqn:A-bar}
\end{equation}
We use $\alpha(p)$ to denote the symmetric $(M+1,M+1,M+1)$-shaped tensor whose $mno$ entry is given by
\begin{equation}
  [\alpha(p)]_{mno} = \sum_{i=1}^M p_i\left(\frac{1}{12}A^s_{mi}A^s_{ni}A^s_{oi} + \frac{1}{6}A^l_{mi}A^l_{ni}A^l_{oi} + \frac{1}{6}A^r_{mi}A^r_{ni}A^r_{oi}\right).
\label{eqn:third-mom-alpha}
\end{equation}

\section{Proof of Proposition \ref{prop:weak-gradient-condition}}
\label{sec:weak-gradient-condition-proofs}
In this section, we provide a full proof of Proposition \ref{prop:weak-gradient-condition}. We prove (\ref{eqn:weak-C-condition}) in \S\ref{sec:C-from-pstar} and (\ref{eqn:weak-p-condition}) in \S\ref{sec:p-from-Cstar}.

\subsection{Recovery of $C$ from $p^*$ (Proof of Equation (\ref{eqn:weak-C-condition}))}
\label{sec:C-from-pstar}

Let $\delta\in\bbR^{(M+1)\times d}$ be sufficiently small and $C = C^* +\delta$ be some other curve in a neighborhood of $C^*$. We wish to show that for $C$ sufficiently close to $C^*$, we have
\[\inner{\grad_C L_{m_3^*}(C, p^*)}{C^*-C} < 0\]
Using (\ref{eqn:loss-jacs-applied}), the left hand side of (\ref{eqn:weak-C-condition}) becomes
{\small \begin{equation}
   2\left\langle(C^{\otimes 3}-{C^*}^{\otimes 3})\star\alpha(p^*), (C\otimes C\otimes C^* + C\otimes C^* \otimes C + C^*\otimes C\otimes C - 3C^{\otimes 3})\star\alpha(p^*)\right\rangle
  \label{eqn:L3-CstarminusC-expanded}
\end{equation}}
We perform a Taylor expansion, replacing $C$ with $C^* + \delta$, expanding the nonlinear terms and only keeping terms with one $\delta$. Then (\ref{eqn:L3-CstarminusC-expanded}) has first order approximation given by
\begin{equation}
  -2\left\|(C^*\otimes C^* \otimes \delta + C^*\otimes\delta\otimes C^* + \delta\otimes C^*\otimes C^*)\star\alpha(p^*) \right\|_F^2.
  \label{eqn:L3-CstarminusC-taylor}
\end{equation}
Let $g(\delta)$ be the term inside the norm: 
\begin{equation}
  g(\delta) \defeq (C^*\otimes C^* \otimes \delta + C^*\otimes\delta\otimes C^* + \delta\otimes C^*\otimes C^*)\star\alpha(p^*).
  \label{eqn:g-defn}
\end{equation}
This is a linear operator from $\bbR^{(M+1)\times d}\to\bbR^{d\times d\times d}$. Therefore, to show (\ref{eqn:weak-C-condition}), it suffices to show that the kernel of $g$ is trivial.

Let $\Sym:\bbR^{d\times d\times d}\to \bbR^{d\times d\times d}$ denote the symmetrization operator on three-way tensors (see (\ref{eqn:sym-def})). Then by the definition of the $\star$ operation and the symmetry of $\alpha(p^*)$, we can rewrite (\ref{eqn:g-defn}) as
\begin{equation}
  g(\delta) =  \Sym\left(3(\delta\otimes C^*\otimes C^*)\star\alpha(p^*)\right).
  \label{eqn:g-with-sym}
\end{equation}
Let $\tilde{g}$ denote the term inside the $\Sym$:
\begin{equation}
  \tilde{g}(\delta) \defeq 3(\delta\otimes C^*\otimes C^*)\star\alpha(p^*).
  \label{eqn:gtilde}
\end{equation}
To show that $g(\delta)$ has trivial kernel, it suffices to show that $(a)$ $\tilde{g}$ has trivial kernel and (b) $\Imag(\tilde{g})\cap \Ker(\Sym) = \{0\}$. We prove (a) in Lemma \ref{lem:gtilde-trivial-kernel} and (b) in Lemma \ref{lem:kernels-have-trivial-intersection}.

\begin{lem}
Let $M = d$. Let $C^* = [c^*_0,\dots, c^*_M]^T$ and assume that $c^*_0,\dots, c^*_M\in\bbR^d$ are in generic position so that any collection of $M$ of the $c^*_i$'s is linearly independent. Let $\tilde{g}$ be the linear operator given in (\ref{eqn:gtilde}). Then $\tilde{g}$ has trivial kernel.
\label{lem:gtilde-trivial-kernel}
\end{lem}
\begin{proof}
Note that $\tilde{g}(\delta)$ is a linear operator from $\bbR^{(M+1)\times d}\to\bbR^{d\times d\times d}$. If we think of this as a matrix operating on the flattened spaces, then with our assumption of $M = d$, the matrix representation of $\tilde{g}$ is a tall and thin matrix. To show that $\tilde{g}$ has trivial kernel, it suffices to show that the $(M+1)d$ ``columns'' of $\tilde{g}$ are linearly independent.

Let $E^{m, j}$ be the standard basis matrix in $\bbR^{(M+1)\times d}$ with all zeros except a $1$ in the $m, j$ entry; note that $m$ ranges from $0$ to $M$ while $j$ ranges from $1$ to $d$, i.e. we zero-index the rows but one-index the columns. Then the $(m, j)$ ``column'' of $\tilde{g}$ is given by $\tilde{g}(E^{m, j})$. Let $E^{m,j}_i$ denote the $i$th row of $E^{m,j}$, which is equal to $e_j$ (the $j$th standard basis vector in $\bbR^d$) if $i = m$ and zero otherwise. Then 
\begin{align*}
  \tilde{g}(E^{m, j}) &= \sum_{i=1}^M p_i^*\Bigg[\frac{1}{4} (E_{i-1}^{m, j} + E_i^{m, j})\otimes (c^*_{i-1} + c^*_i)\otimes (c^*_{i-1} + c^*_i) \\
  &\hspace{6em}+ \frac{1}{2}E_{i-1}^{m, j}\otimes c^*_{i-1}\otimes c^*_{i-1} + \frac{1}{2}E_i^{m, j}\otimes c^*_i \otimes c^*_i\Bigg]\\
  &= \frac{1}{4}p_m^* e_j\otimes (c^*_{m-1} + c^*_m) \otimes (c^*_{m-1} + c^*_m) + \frac{1}{2}p_m^* e_j\otimes c^*_m \otimes c^*_m \\
  &\hspace{2em} + \frac{1}{4}p_{m+1}^*e_j\otimes (c^*_m + c^*_{m+1})\otimes (c^*_m+c^*_{m+1}) + \frac{1}{2}p_{m+1}^*e_j\otimes c^*_m\otimes c^*_m\\
  &= \left[0, \dots, 0, F^m, 0, \dots, 0\right].
\end{align*}
In the last line, we are notating a $(M+1, M+1, M+1)$-shaped tensor as a size $M+1$ vector consisting of size $(M+1, M+1)$ matrices, where the matrix $F^m$ is in the $j$th slot of the vector and is given by 
{\small\begin{equation}
  F^m = \begin{cases}
    \frac{1}{4}p_1^*(c^*_0 + c^*_1)^{\otimes 2} + \frac{1}{2}p_1^*{c^*_0}^{\otimes 2} & \text{if } m = 0\\
    \frac{1}{4}p_m^*(c^*_{m-1}+c^*_m)^{\otimes 2} + \frac{1}{4}p_{m+1}^*(c^*_m+c^*_{m+1})^{\otimes 2} + \frac{1}{2}p_m^*{c^*_m}^{\otimes 2} + \frac{1}{2}p_{m+1}^*{c^*_m}^{\otimes 2} &\text{if }1\leq m\leq M-1\\
    \frac{1}{4}p_M^*(c^*_{M-1} + c^*_M)^{\otimes 2} + \frac{1}{2}p_M^*{c^*_M}^{\otimes 2} &\text{if }m = M
  \end{cases}
  \label{eqn:Fm}
\end{equation}}
From the position of the nonzero entries, it is clear that $\tilde{g}(E^{m, j})$ are linearly independent for fixed $m$ and varying $j$. It remains to show that there is also linear independence for fixed $j$ and varying $m$; in other words, we need to show that $\{F^m\}_{m=0}^M$ is a linearly independent collection of matrices.

To show this, we wish to show that if $T\defeq \sum_{m=0}^M k_mF^m = 0$, then $k_m = 0$ for all $m = 0, \dots, M$. We can reindex the sum to write 
\begin{equation}
  T = \sum_{i=1}^M\frac{k_{i-1} + k_i}{4}p_i(c_{i-1}+c_i)^{\otimes 2} + \frac{k_{i-1}}{2}p_ic_{i-1}^{\otimes 2} +\frac{k_i}{2}p_ic_i^{\otimes 2}
  \label{eqn:T-reindexed}
\end{equation}
For $j=2,\dots, M-1$, let $x_j$ be a vector in $\mathbb{R}^d$ such that $x_j\perp c_i$ for all $i\neq j-1, j$. Let $Tx_j$ denote the tensor contraction along the last index of $T$. Then $Tx_j$ kills a large number of terms and we have 
\begin{align*}
    Tx_j&= \frac{k_{j-2}+k_{j-1}}{4}p_{j-1}\langle c_{j-1}, x_j\rangle (c_{j-2}+c_{j-1}) \\
    &\hspace{3em}+ \frac{k_{j-1}+k_{j}}{4}p_{j}\langle c_{j-1}+c_{j}, x_j\rangle (c_{j-1}+c_{j})\\
    &\hspace{3em}+ \frac{k_{j}+k_{j+1}}{4}p_{j+1}\langle c_{j}, x_j\rangle (c_{j}+c_{j+1})\\
    &\hspace{3em}+ \frac{k_{j-1}}{2}p_j\langle c_{j-1}, x\rangle c_{j-1} + \frac{k_j}{2}p_{j+1}\langle c_j, x\rangle c_j\\
    &\hspace{3em}+ \frac{k_{j-1}}{2}p_{j-1}\langle c_{j-1}, x\rangle c_{j-1} + \frac{k_j}{2}p_{j}\langle c_j, x\rangle c_j
\end{align*}
We can rewrite this as $Tx_j = \alpha_{j-2}c_{j-2} + \alpha_{j-1}c_{j-1} + \alpha_{j}c_{j} + \alpha_{j+1}c_{j+1}$, where 
\begin{align*}
\alpha_{j-2} &= \frac{k_{j-2}+k_{j-1}}{4}p_{j-1}\langle c_{j-1}, x_j\rangle\\
\alpha_{j-1} &= \frac{k_{j-2}+k_{j-1}}{4}p_{j-1}\langle c_{j-1}, x_j\rangle + \frac{k_{j-1}+k_{j}}{4}p_{j}\langle c_{j-1}+c_{j}, x_j\rangle \\
&\hspace{3em}+ \frac{k_{j-1}}{2}p_j\langle c_{j-1}, x\rangle + \frac{k_{j-1}}{2}p_{j-1}\langle c_{j-1}, x\rangle\\
\alpha_{j} &= \frac{k_{j-1}+k_{j}}{4}p_{j}\langle c_{j-1}+c_{j}, x_j\rangle + \frac{k_{j}+k_{j+1}}{4}p_{j+1}\langle c_{j}, x_j\rangle  + \frac{k_j}{2}p_{j+1}\langle c_j, x\rangle + \frac{k_j}{2}p_{j}\langle c_j, x\rangle\\
\alpha_{j+1} &= \frac{k_{j}+k_{j+1}}{4}p_{j+1}\langle c_{j}, x_j\rangle
\end{align*}
By our assumption of genericness, $c_{j-2}, \dots, c_{j+1}$ are linearly independent. Note also that $x_j$ is not in the span of $c_{j-2}$ and $c_{j+1}$ by our definition of $x_j$. Therefore, if the linear combination of the four terms above is to be zero, that means we must have $\alpha_{j-2}, \alpha_{j+1} = 0$. Our genericness assumption also implies that $\langle c_{j-1}, x_j\rangle$ and $\langle c_j, x_j\rangle$ are nonzero. Note also that all $p_i$ are nonzero (positive, in fact). Therefore, if $\alpha_{j-2}, \alpha_{j+1} = 0$, then it must follow that $k_{j-2} + k_{j-1} = 0$ and $k_{j} + k_{j+1} = 0$. Repeat this for $j = 2, \dots, M-1$. Then if $\sum_{m=0}^M k_mF^m = 0$, then (\ref{eqn:T-reindexed}) becomes
\[\sum_{i=1}^M \left(\frac{k_{i-1}}{2}p_ic_{i-1}^{\otimes 2} + \frac{k_i}{2}p_ic_i^{\otimes 2}\right) = 0.\]
We have thus reduced our problem to whether $c_0^{\otimes 2},\dots, c_M^{\otimes 2}$ is a linearly independent set. 

Consider the sum
\begin{equation}
\sum_{i=0}^M \gamma_ic_i^{\otimes 2} 
\end{equation}
where $\gamma_i$ are scalars. We can rewrite this as 
\begin{equation}
\gamma_0c_0^{\otimes 2} + \sum_{i=1}^M\gamma_ic_i^{\otimes 2}.
\label{eqn:gamma-sum-decomp}
\end{equation}
Suppose for the sake of contradiction that (\ref{eqn:gamma-sum-decomp}) is equal to zero and not all $\gamma_i$ are equal to zero; assume without loss of generality that $\gamma_0$ is nonzero. Since $c_1,\dots,c_M$ are linearly independent, the summation is rank $r$, where $r$ is the number of nonzero coefficients $\gamma_1,\dots,\gamma_M$, so (\ref{eqn:gamma-sum-decomp}) is the sum of a rank 1 and rank $r$ matrix. If $r>1$, then this sum cannot possibly be zero. If $r=1$, suppose without loss of generality that $\gamma_1$ is the only nonzero coefficient among $\gamma_1,\dots,\gamma_M$, so that $\gamma_0c_0^{\otimes 2} + \gamma_1 c_1^{\otimes 2} = 0$. Then $c_0$ must be parallel to $c_1$, but this is not possible by our assumption that $c_i$ are generic.
\end{proof}

Before we show  $\Imag(\tilde{g})\cap \Ker(\Sym) = \{0\}$, we first need the following result, which we have adapted from Lemma 5.6 in \cite{determinant}.
\begin{lem}
\label{lem:determinant}
Let $\calM(x)$ be a matrix whose entries are analytic functions of a variable $x\in U\subset\bbR^n$ for some connected open domain $U$. Let $r$ be the rank of $\calM(x_0)$ for some particular $x_0\in \bbR^n$. Then $\rank(\calM(x))\geq r$ for almost all $x\in U$.
\end{lem}
\begin{proof}
Note that $\rank(\calM(x))<r$ if and only if the determinant of every $(r, r)$-shaped minor of $\calM(x)$ vanishes. Since the rank of $\calM(x_0)$ is $r$, there exists some $(r,r)$-shaped minor $\hat{\calM}(x)$ such that $\det(\hat{\calM}(x_0))\neq 0$. Note that $\det(\hat{\calM}(x))$ is an analytic function of $x$, and the witness $x_0$ shows that $\det(\hat{\calM}(x))$ is not identically zero. An analytic function that is not identically zero has a measure zero vanishing set (see \cite{analytic-zeros}). Therefore, $\det(\hat{\calM}(x))$ has a vanishing set of measure zero, and hence $\rank(\calM(x))$ is $<r$ on a measure zero set.
\end{proof}
The utility of this lemma is that if a matrix $\calM(x)$ has entries that are analytic and we find a single witness $x_0$ such that $\calM(x_0)$ is full rank, then we know $\calM(x)$ is full rank for almost every $x$. We use this lemma in the proof of the following result.

\begin{lem}
Let $d \geq 4$ and $M=d$. Then for almost all $C^*\in\bbR^{(M+1)\times d}$, the image of $\tilde{g}$ and the kernel of $\Sym$ are linear subspaces that intersect only at $0$.
\label{lem:kernels-have-trivial-intersection}
\end{lem}
\begin{proof}
In the proof of Lemma \ref{lem:gtilde-trivial-kernel}, we showed that $\left\{\tilde{g}(E^{m, j})\right\}_{0\leq m\leq M, 1\leq j\leq d}$ forms a basis for $\Imag(\tilde{g})$. Since the action of $\tilde{g}$ depends on $C^*$, in this proof we will use the notation $\tilde{g}_{C^*}$ to denote the $\tilde{g}$ operator induced by a particular choice of $C^*$. For coefficients $\beta_{m, j}$, we wish to show that if 
\begin{equation}
  \Sym\left(\sum_{\substack{0\leq m\leq M \\ 1\leq j\leq d}}\beta_{m, j} \tilde{g}_{C^*}(E^{m, j})\right) = 0,
\end{equation}
then all the coefficients must be zero. By linearity of the $\Sym$ operator, it suffices to show that 
\begin{equation}
  \left\{\Sym(\tilde{g}_{C^*}(E^{m, j}))\right\}_{0\leq m\leq M, 1\leq j\leq d}
  \label{eqn:Sym-gtilde-basis}
\end{equation} 
forms a linearly independent set of $(M,M,M)$-shaped tensors.

Let $\calM(C^*)$ denote the $(M^3, (M+1)M)$-shaped matrix of column-stacked flattenings of the tensors in (\ref{eqn:Sym-gtilde-basis}). Then we wish to show that this matrix has full column rank for almost all $C^*$. The entries of $\calM(C^*)$ are either zeros or linear combinations of the entries of $F^m$ as defined in (\ref{eqn:Fm}). Recall that the values $p_i^*$ are a function of $C^*$ in the following way:
\begin{equation}
  \label{eqn:p-from-C}
  p_i^* =  p_i^*(C^*) = \frac{\|c_i^* - c_{i-1}^*\|}{\sum_{j=1}^M \|c_j^* - c_{j-1}^*\|}.
\end{equation}
On the connected open domain $U \defeq \{C^*\in\bbR^{(M+1)\times d}: c^*_{j} - c^*_{j-1}\neq 0, \;j=1,\dots M\}$, this is an analytic function of the entries of $C^*$. Since the entries of terms of the form $(c^*_{m-1} + c^*_{m})^{\otimes 2}$ and ${c^*_m}^{\otimes 2}$ are also analytic, it follows from the form of $F^m$ that the entries of $\calM(C^*)$ are analytic functions of the entries of $C^*$. By Lemma \ref{lem:determinant}, it suffices then to find a single witness $C^*$ for each dimension $M$ where $\calM(C^*)$ has full rank.

For $M\geq 4$, we define the following witness $C^M$ (for notational clarity, we show the case $M=4$ here, with $C^M$ for greater $M$ being defined in the obvious way):
\begin{equation}
  \label{eqn:CM}
  C^M\defeq \begin{bmatrix}\frac{1}{\sqrt{2}} & 0 & 0 \\ 0 & \frac{1}{\sqrt{2}} & 0 \\ 0 & 0 & \frac{1}{\sqrt{2}} \\ 0 & 0 & \frac{1}{\sqrt{2}} + 1\end{bmatrix}.
\end{equation}
We choose this particular $C^M$ because it induces matrices $F^m$ that are sparse and have a particular recursive structure, and the proportional segment lengths of $C^M$ are all $1/M$. Define
\begin{align}
  \label{eqn:WM}
  \begin{split}
  w_M^{m, j} &\defeq M\Sym(\tilde{g}_{C^M}(E^{m, j}))\\
  W^M &= \{w_M^{m, j}\}_{0\leq m\leq M, 1\leq j\leq M}.
  \end{split}
\end{align}
To show that $C^M$ is a witness for a full rank $\calM(C^*)$, it suffices to show that for all $M\geq 4$, $W^M$ is a linearly independent collection of tensors. We proceed by induction. The base case $M=4$ can be verified by straightforward computation. We verify this numerically by computing the singular value decomposition of the matrix $\calM(C^4)$ and observing that the smallest singular value is positive; the smallest singular value is approximately $3.29024\times 10^{-4}$. 

Now suppose $W^M$ is a linearly independent collection; we wish to show that $W^{M+1}$ also is a linearly independent collection. Note that $w_{M+1}^{m, j}$ is an $(M+1, M+1, M+1)$-shaped tensor. Let $\hat{w}_{M+1}^{m, j}$ denote the $(M, M, M)$-shaped ``minor'' of $w_{M+1}^{m, j}$ given by the $(1:M, 1:M, 1:M)$ subtensor of $w_M^{m, j}$ (here we are $1$-indexing). By straightforward explicit computation of the entries of $M\Sym(\tilde{g}_{C^{M}}(E^{m, j}))$ and $(M+1)\Sym(\tilde{g}_{C^{M+1}}(E^{m, j}))$, it can be shown that the minors of $w_{M+1}^{m, j}$ have the following properties:
\begin{enumerate}[(a)]
  \item For $0\leq m\leq M-2$, $1\leq j\leq M$: $\hat{w}_{M+1}^{m, j} = w_M^{m, j}$.
  \item For $1\leq j\leq M$: $\displaystyle \hat{w}_{M+1}^{M-1, j} = w_M^{M-1, j} - \left(\frac{1}{\sqrt{2}} - \frac{1}{4}\right)w_M^{M, j}$.
  \item For $1\leq j \leq M$: $\displaystyle  \hat{w}_{M+1}^{M, j} = \frac{1}{4(1+\sqrt{2})^2}w_M^{M, j}$.
  \item For $1\leq j\leq M$: $\hat{w}_{M+1}^{M+1, j} = 0$.
  \item For $0\leq m \leq M+1$: $\hat{w}_{M+1}^{m, M+1} = 0$.
\end{enumerate}
Now given coefficients $\alpha_{m, j}$, suppose
\begin{equation}
  \sum_{\substack{0\leq m \leq M+1 \\ 1\leq j \leq M+1}}\alpha_{m, j}w_{M+1}^{m, j} = 0.
  \label{eqn:w-lc}
\end{equation}
To complete the inductive step, it suffices to show that all coefficients $\alpha_{m, j}$ must be zero. Observe that the $(1:M, 1:M, 1:M)$ subtensor of the linear combination above must also be zero, and hence using the properties above we have 
\begin{align*}
  0 &= \sum_{\substack{0\leq m \leq M+1 \\ 1\leq j \leq M+1}}\alpha_{m, j}\hat{w}_{M+1}^{m, j} \\
  &= \sum_{\substack{0\leq m \leq M-2 \\ 1\leq j \leq M}} \alpha_{m, j}\hat{w}_{M+1}^{m, j} + \sum_{1\leq j \leq M}\alpha_{M-1, j}\hat{w}_{M+1}^{M-1, j} + \sum_{1\leq j \leq M}\alpha_{M, j}\hat{w}_{M+1}^{M, j} \\
  &\hspace{3em}+ \sum_{1\leq j\leq M}\alpha_{M+1, j}\hat{w}_{M+1}^{M+1, j} + \sum_{0\leq m\leq M+1}\alpha_{m, M+1}\hat{w}_{M+1}^{m, M+1}\\
  &= \sum_{\substack{0\leq m \leq M-2 \\ 1\leq j \leq M}} \alpha_{m, j}w_{M}^{m, j} + \sum_{1\leq j \leq M}\alpha_{M-1, j}\left(w_M^{M-1, j} - \left(\frac{1}{\sqrt{2}} - \frac{1}{4}\right)w_M^{M, j}\right)\\
  &\hspace{3em}+ \sum_{1\leq j \leq M}\alpha_{M, j}\frac{1}{4(1+\sqrt{2})^2}w_M^{M, j}\\
  &=\sum_{\substack{0\leq m \leq M-2 \\ 1\leq j \leq M}} \alpha_{m, j}w_{M}^{m, j} + \sum_{1\leq j \leq M}\alpha_{M-1, j}w_M^{M-1, j}\\
  &\hspace{3em} + \sum_{1\leq j \leq M}\left(\alpha_{M, j}\frac{1}{4(1+\sqrt{2})^2} - \alpha_{M-1, j}\left(\frac{1}{\sqrt{2}} - \frac{1}{4}\right)\right)w_M^{M, j}.
\end{align*}
This is a linear combination of the terms $\{w_M^{m, j}\}_{0\leq m \leq M, 1\leq j \leq M} = W^M$; by the inductive hypothesis, this is a set of linearly independent tensors. In the last line above, we see that all $\alpha_{m,j}$ in the first sum must be zero and all $\alpha_{M-1, j}$ in the second sum must be zero. Therefore all $\alpha_{M-1, j}$ in the third sum of the last line must be zero, and hence all $\alpha_{M,j}$ in the third sum must be zero as well. We can thus conclude that in (\ref{eqn:w-lc}), all $\{\alpha_{m,j}\}_{0\leq m\leq M, 1\leq j\leq M}$ must be zero. Therefore, most of the terms in (\ref{eqn:w-lc}) vanish, and we are left with
\begin{equation}
  \label{eqn:w-lc-tail}
  \sum_{\substack{0\leq m \leq M+1 \\ 1\leq j \leq M+1}}\alpha_{m, j}w_{M+1}^{m, j} = \sum_{0\leq m\leq M}\alpha_{m, M+1}w_{M+1}^{m, M+1} + \sum_{1\leq j \leq M+1}\alpha_{M+1, j}w_{M+1}^{M+1, j} = 0.
\end{equation}
Now consider the $(M+1, :, :)$ slice of the above tensor equation, which must equal the zero matrix. For $M+1 = 5$, the $(M+1, :, :)$ slices of $\{w_{M+1}^{m, M+1}\}_{0\leq m\leq M}$ and $\{w_{M+1}^{M+1, j}\}_{1\leq j \leq M+1}$ are as follows, with the slices for greater values of $M+1$ being the obvious extension of the below (with the nonzero values being the same and more zero padding).
{\scriptsize \begin{equation}
  \label{eqn:slices}
  \begin{split}
    w_{M+1}^{0, M+1}[M+1, :, :] &= \begin{bmatrix} 1/8 & 1/24 & 0 & 0 & 0\\ 1/24 & 1/24 & 0 & 0 & 0 \\ 0 & 0 & 0 & 0 & 0\\ 0 & 0 & 0 & 0 & 0\\ 0 & 0 & 0 & 0 & 0\end{bmatrix}\;\;\\
    w_{M+1}^{1, M+1}[M+1, :, :] &= \begin{bmatrix} 1/24 & 1/24 & 0 & 0 & 0\\ 1/24 & 1/4 & 1/24 & 0 & 0 \\ 0 & 1/24 & 1/24 & 0 & 0\\ 0 & 0 & 0 & 0 & 0\\ 0 & 0 & 0 & 0 & 0\end{bmatrix}\;\;\\
    w_{M+1}^{2, M+1}[M+1, :, :] &= \begin{bmatrix} 0 & 0 & 0 & 0 & 0 \\ 0 & 1/24 & 1/24 & 0 & 0\\ 0 & 1/24 & 1/4 & 1/24  & 0 \\ 0 & 0 & 1/24 & 1/24 & 0\\ 0 & 0 & 0 & 0 & 0\end{bmatrix}\;\;\\
    w_{M+1}^{3, M+1}[M+1, :, :] &= \begin{bmatrix} 0 & 0 & 0 & 0 & 0 \\ 0 & 0 & 0 & 0 & 0 \\  0 & 0 & 1/24 & 1/24  & 0 \\ 0 & 0 & 1/24 & 1/4 & 1/12\\ 0 & 0 & 0 & 1/12 & 1/8\end{bmatrix}\;\;\\
    w_{M+1}^{4, M+1}[M+1, :, :] &= \begin{bmatrix} 0 & 0 & 0 & 0 & 0 \\ 0 & 0 & 0 & 0 & 0 \\  0 & 0 & 0 & 0 & 0 \\ 0 & 0 & 0 & 1/24 & 1/12\\ 0 & 0 & 0 & 1/12 & \frac{11}{8}+\frac{1}{\sqrt{2}}\end{bmatrix}\;\;
  \end{split}
  \begin{split}
    w_{M+1}^{M+1, 1}[M+1, :, :] &= \begin{bmatrix} 0 & 0 & 0 & 0 & \frac{(1+\sqrt{2})^2}{6} \\ 0 & 0 & 0 & 0 & 0 \\ 0 & 0 & 0 & 0 & 0 \\ 0 & 0 & 0 & 0 & 0 \\ \frac{(1+\sqrt{2})^2}{6} & 0 & 0 & 0 & 0\end{bmatrix}\\
    w_{M+1}^{M+1, 2}[M+1, :, :] &= \begin{bmatrix} 0 & 0 & 0 & 0 & 0 \\ 0 & 0 & 0 & 0 & \frac{(1+\sqrt{2})^2}{6} \\ 0 & 0 & 0 & 0 & 0 \\ 0 & 0 & 0 & 0 & 0 \\ 0 & \frac{(1+\sqrt{2})^2}{6} & 0 & 0 & 0\end{bmatrix}\\
    w_{M+1}^{M+1, 3}[M+1, :, :] &= \begin{bmatrix} 0 & 0 & 0 & 0 & 0\\ 0 & 0 & 0 & 0 & 0 \\ 0 & 0 & 0 & 0 & \frac{(1+\sqrt{2})^2}{6}  \\ 0 & 0 & 0 & 0 & 0 \\ 0 & 0 & \frac{(1+\sqrt{2})^2}{6} & 0 & 0\end{bmatrix}\\
    w_{M+1}^{M+1, 4}[M+1, :, :] &= \begin{bmatrix} 0 & 0 & 0 & 0 & 0 \\ 0 & 0 & 0 & 0 & 0 \\ 0 & 0 & 0 & 0 & 0 \\ 0 & 0 & 0 & 0 & \frac{(1+\sqrt{2})^2}{6} \\ 0 & 0 & 0 & \frac{(1+\sqrt{2})^2}{6} & 0\end{bmatrix}\\
    w_{M+1}^{M+1, 5}[M+1, :, :] &= \begin{bmatrix} 0 & 0 & 0 & 0 & 0\\ 0 & 0 & 0 & 0 & 0 \\ 0 & 0 & 0 & 0 & 0 \\ 0 & 0 & 0 & 0 & 0 \\ 0 & 0 & 0 & 0 & \frac{(1+\sqrt{2})^2}{2}\end{bmatrix}
  \end{split} 
\end{equation}}
One can show that for all dimensions $M$, these matrices are all linearly independent, and hence the remaining coefficients in (\ref{eqn:w-lc-tail}) must also all be zero. Therefore, all coefficients in (\ref{eqn:w-lc}) must be zero, and hence $W^{M+1}$ is a linearly independent collection, as desired.
\end{proof}
This completes the proof of (\ref{eqn:weak-C-condition}) holding for almost all $C^*$. 

\subsection{Recovery of $p$ from $C^*$ (Proof of Equation (\ref{eqn:weak-p-condition}))}
\label{sec:p-from-Cstar}
Now we wish to show that for $p$ sufficiently close to $p^*$, we have (\ref{eqn:weak-p-condition}), which we repeat here for convenience:
\[\inner{\grad_p L_{m_3^*}(C^*, p)}{p^*-p}<0.\]
We can actually show a stronger statement, that the above is true for \textit{all} $p\neq p^*$. Using (\ref{eqn:loss-jacs-applied}) and the fact that $\alpha$ is linear in $p$, we have 
\begin{align*}\
  \inner{\grad_p L_{m_3^*}(C^*, p)}{p^*-p} &= \inner{2{C^*}^{\otimes 3} \star\alpha(p-p^*)}{{C^*}^{\otimes 3}\star\alpha(p^*-p)}\\
  &= -2\|{C^*}^{\otimes 3}\star \alpha(p - p^*)\|^2.
\end{align*}
Unpacking the definition of $\alpha$ and the $\star$ operator, the term inside the norm on the last line is equal to 
\begin{equation}
\sum_{i=1}^M (p_i - p_i^*)\left(\frac{1}{12}(c_{i-1}^* + c_i^*)^{\otimes 3} + \frac{1}{6}(c_{i-1}^*)^{\otimes 3} + \frac{1}{6}(c_{i}^*)^{\otimes 3}\right).
\end{equation} 
It suffices to show that this is equal to zero if and only if $p_i - p_i^* = 0$ for all $i$. In other words, we need to show that  
\begin{equation}
  \left\{\frac{1}{12}(c_{i-1}^* + c_i^*)^{\otimes 3} + \frac{1}{6}(c_{i-1}^*)^{\otimes 3} + \frac{1}{6}(c_{i}^*)^{\otimes 3}\right\}_{i=1}^M
  \label{eqn:p-indep-tensors}
\end{equation}
is a collection of $M$ linearly independent three-way tensors for $c_i^*$ in generic position. We omit the proof of this because it uses techniques very similar our proof of the independence of $\{F^m\}_{m=0}^M$ in the proof of Lemma \ref{lem:gtilde-trivial-kernel}. This completes the proof of (\ref{eqn:weak-p-condition}) and hence the proof of Proposition \ref{prop:weak-gradient-condition}.

\section{The Basin of Attraction in Corollary \ref{cor:m3-well-posed}}
\label{sec:basin-of-attraction}
Here we briefly investigate the size of the basin of attraction $U$ in Corollary \ref{cor:m3-well-posed}. To do so, we randomly sample four ground truth curves $C^*$ with $M = d = 32$ and compute their third moments $m_3^*$. We then sample a Gaussian noising term $\xi$ with the same shape (as matrices) as $C^*$, then perturb each $C^*$ with a scaling of $\xi$ to get an initialization for gradient descent $\hat{C}_\textrm{init} \defeq C^* + s\xi$; the scale $s$ takes values in $\{0.1, 0.2, \dots, 0.5\}$. We compute $\hat{p}_{\textrm{init}}$ from $\hat{C}_\textrm{init}$, perform alternating gradient descent on the loss $L_{m_3^*}$ with $\hat{C}_\textrm{init}$ and $\hat{p}_\textrm{init}$ as initializations, and study the resulting $\hat{C}$; we plot $C^*$, $\hat{C}_\textrm{init}$, and $\hat{C}$ for various scales in Figure \ref{fig:basin-of-attraction}. We see that qualitatively, the alternating gradient descent is fairly robust under moderate perturbations, breaking at noise scale around $s=0.4$. In Figure \ref{fig:basin-of-attraction-line-plot}, for each ground truth $C^*$, we plot the error of the perturbation $\sqrt{\rho(C^*, \hat{C}_{\textrm{init}})}$ against the error of the result of gradient descent $\sqrt{\rho(C^*, \hat{C})}$ over the five perturbation scales.
\begin{figure}
    \centering
    \includegraphics[width = 0.9\textwidth]{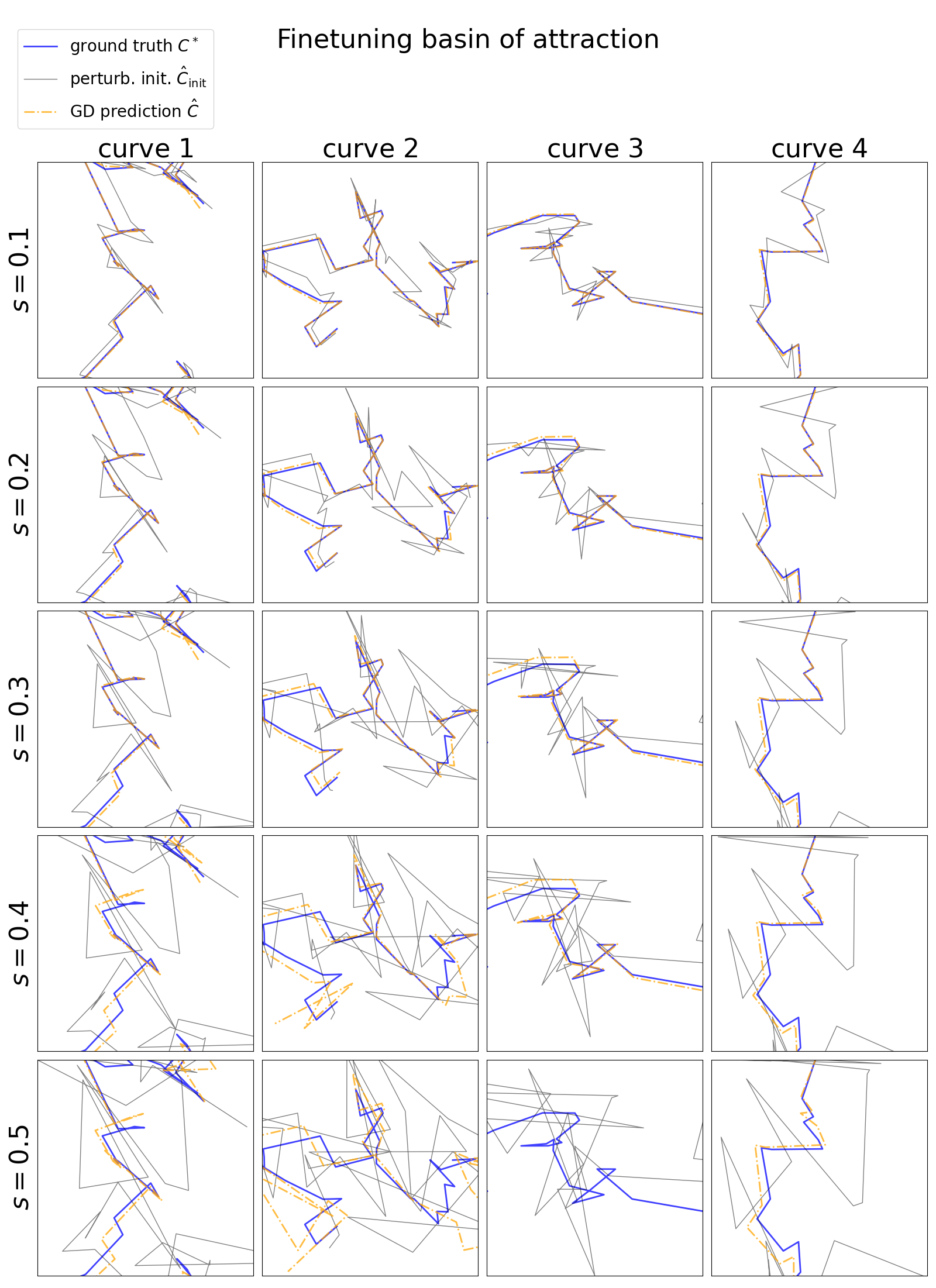} 
    \caption{Each column corresponds to a randomly sampled ground truth curve $C^*$ with $M = d = 32$. Each row corresponds to a different noising level $s$ applied to $C^*$ (blue) to obtain an initialization $\hat{C}_\textrm{init}$ (gray), which we use for gradient descent to obtain a predicted curve $\hat{C}$ (orange). We see that in these cases, the gradient descent fails to converge around $s = 0.4$. Even at noise level $s = 0.3$ where $\hat{C}_\textrm{init}$ is quite perturbed compared to $C^*$, the initialization appears to be within the basin of attraction. There is no gradient descent prediction in the last row for curve 3 due to convergence to NaN from numerical issues.}
    \label{fig:basin-of-attraction}
\end{figure}

\begin{figure}
    \centering
    \includegraphics[width = 0.6\textwidth]{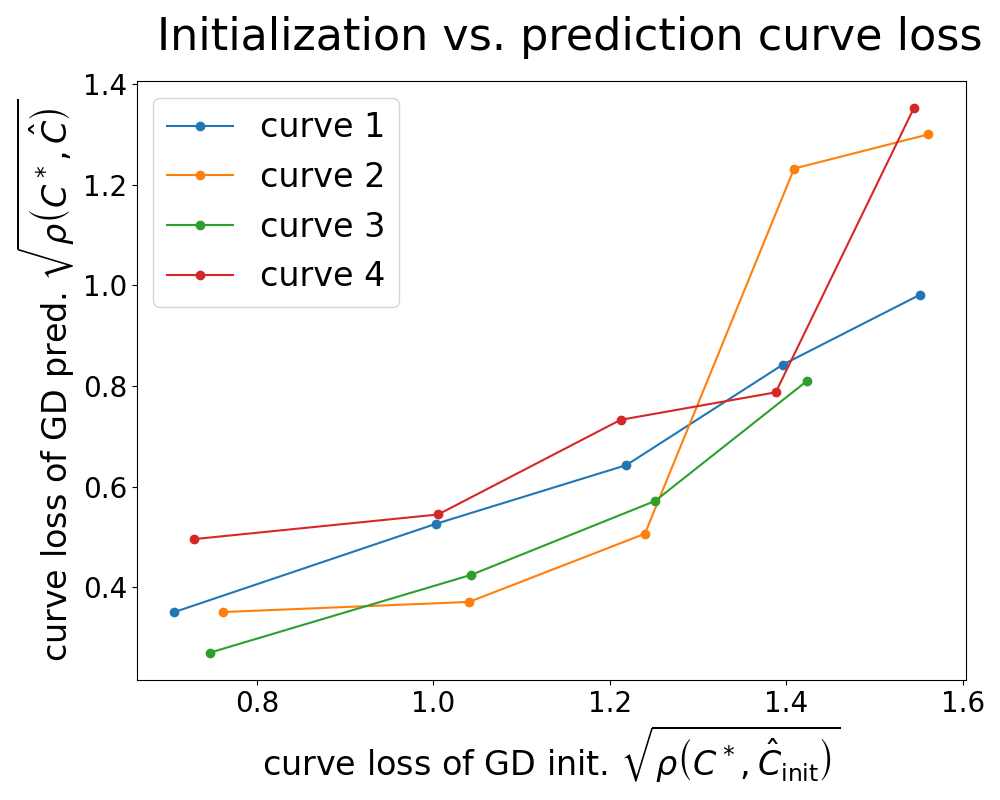}
    \caption{Each line corresponds to a ground truth curve from the columns of Figure \ref{fig:basin-of-attraction} and plots the error of the perturbed gradient descent initialization $\sqrt{\rho\left(C^*, \hat{C}_\textrm{init}\right)}$ versus the error of the result of gradient descent $\sqrt{\rho\left(C^*, \hat{C}\right)}$.}
    \label{fig:basin-of-attraction-line-plot}
\end{figure}

\section{The Heuristic Tensor Power Method Decomposition of the Third Moment}
\label{sec:tpm-cross-terms-ok}
Here we provide some numerical evidence justifying the discarding of cross terms in (\ref{eqn:m3-tpm}) and applying the tensor power method. In Figure \ref{fig:tpm-cross-terms-ok}, we have plotted the ground truth subspaces of the ground truth vertices of curves $C$ (in blue) and the predicted subspaces obtained by our heuristic tensor power decomposition of (in orange). The curves used are the same curves from the experiment in Figure \ref{fig:true-cnst}, using curves sampled in $\mathbb{R}^{48}$ with $M = 32$ segments, with curves sampled by iteratively sampling directions on the unit sphere uniformly randomly. Subplot titles indicate the average discrepancy in degrees between the ground truth and predicted subspaces after applying our heuristic algorithm that sorts the subspaces. We see that in each case, the predicted and ground truth subspaces are approximately aligned, and the average discrepancy between the ground truth and predicted subspaces is roughly 5 degrees. This suggests that discarding the cross terms in the third moment expression (\ref{eqn:m3-tpm}) does not alter the third moment enough to significantly affect its decomposition via the tensor power method.
\begin{figure}
    \includegraphics[width = \textwidth]{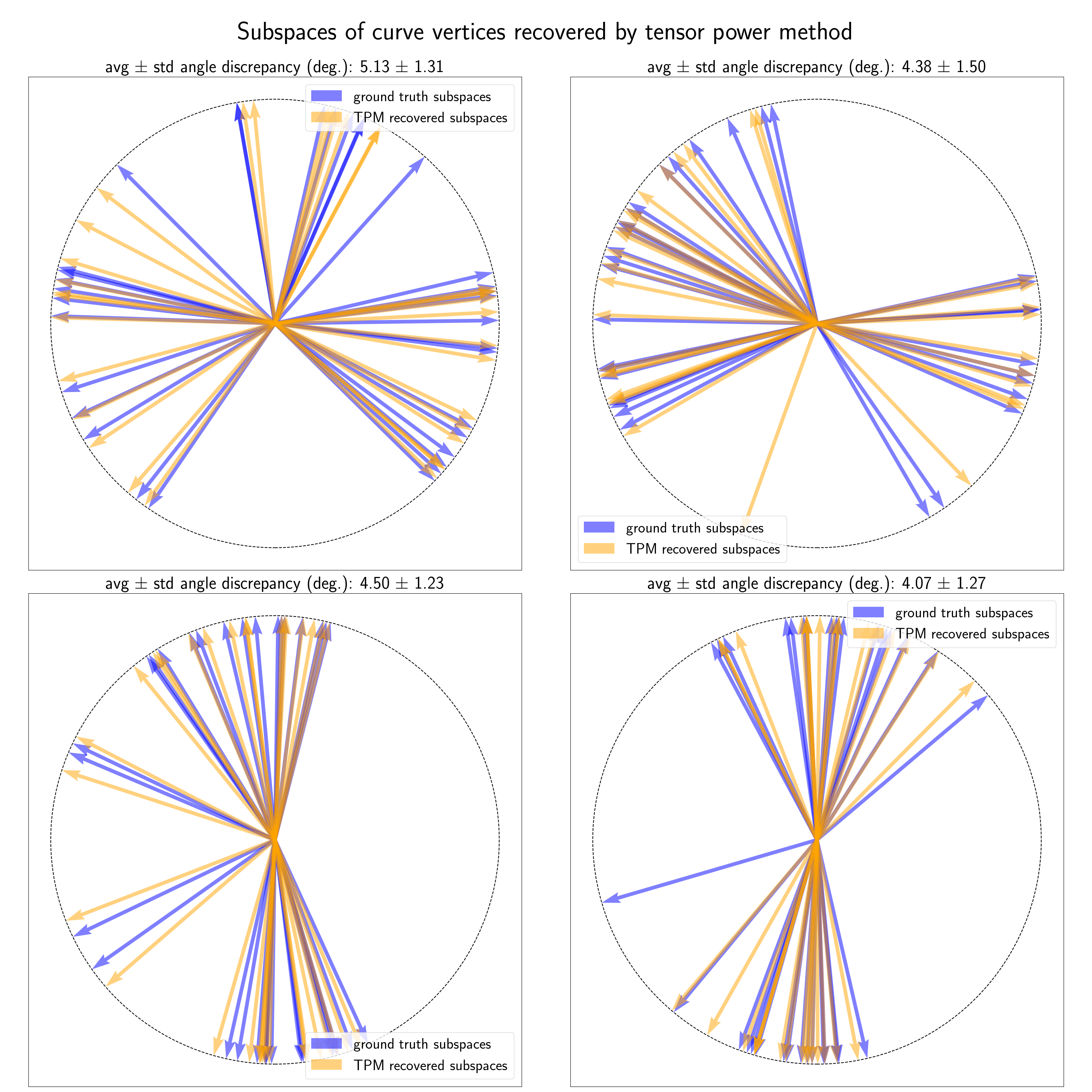}
    \caption{Ground truth subspaces of curve vertices (blue) compared to predicted subspaces obtained by applying the TPM to the third moment tensor (orange). The ground truth curves are the same as those from Figure \ref{fig:true-cnst}. Dashed black lines indicate the unit circle. Subspaces have been projected down to the first two dimensions of $\mathbb{R}^{48}$ for visualization purposes.}
    \label{fig:tpm-cross-terms-ok}
\end{figure}

\section{Estimation of the Third Moment from Data}
\label{sec:third-mom-estimation}
Here, we briefly study the accuracy of estimation of the third moment of a noisy curve from data given varying number of samples from the noisy curve model (\ref{eqn:noisy-curve-model}). In this experiment, we sample four random curves with the same parameters as the experiment in Figure \ref{fig:cloud-cnst} (with ambient dimension $d = 24$, number of segments $M = 16$, noise level $\sigma = 1/4$). We then estimate the moments of these four curves using the unbiased estimators in \ref{lem:unbiased-estimators} from $N$ samples from (\ref{eqn:noisy-curve-model}), where $N = 10^4, 10^5,\dots, 10^{10}$, and compute the average entrywise relative error between the estimated and ground truth third moments. We show these results in Figure \ref{fig:third-mom-estimation-err}. Observe that in three out of four cases, $10^8$ samples (which is the number of samples used in the experiment from Figure \ref{fig:cloud-cnst}) is roughly the number of samples needed to obtain a roughly one percent relative error per entry of the third moment tensor. 

\begin{figure}
    \centering
    \includegraphics[width = 0.7\textwidth]{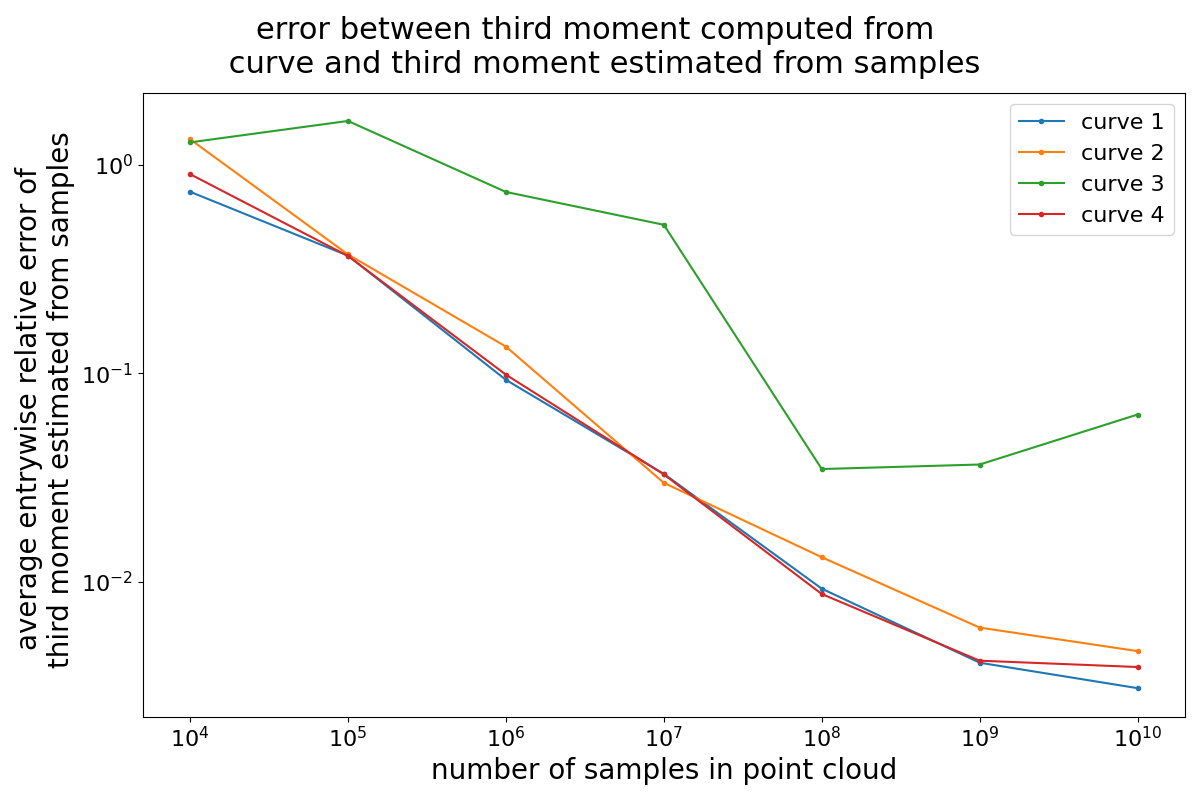}
    \caption{Log-log plot of average entrywise relative error of third moment tensor estimated from data for various sample sizes for four different ground truth curves with $d = 24$, $M = 16$, $\sigma = 1/4$.}
    \label{fig:third-mom-estimation-err}
\end{figure}

\section{Miscellaneous Technical Lemmas}
In this section, we collect miscellaneous technical lemmas and information-theoretic facts.  

\begin{lem}
  \label{lem:tensor-pwr-diff-bound}
  Let $x, y\in\bbR^d$, $\|x\|, \|y\|\leq 1$. Suppose also that we have $\|x - y\|\leq 1/3$. Then for any $k\geq 1$, we have the bound 
  \begin{equation}
    \label{eqn:tensor-pwr-diff-bound}
    \|x^{\otimes k} - y^{\otimes k}\|_F^2 \leq 12\cdot 2^k\|x - y\|^2.
  \end{equation}
\end{lem}
The proof is adapted from the proof of Lemma B.12 in \cite{mra-background}.
\begin{proof}
  Let $\gamma \defeq \inner{x}{y-x}$. Note $|\gamma|\leq \|x\|\|x-y\|\leq\|x-y\|$. Observe that we can write
  \begin{align*}
    \inner{x}{y} &= \|x\|^2 + \gamma\\
    \|y\|^2 &= \|x\|^2 + 2\gamma + \|x - y\|^2.
  \end{align*}  
  By the binomial theorem,
  \begin{align*}
    \|x^{\otimes k} - y^{\otimes k}\|_F^2 &= \|x^{\otimes k }\|^2 - 2\inner{x^{\otimes k}}{y^{\otimes k}} + \|y^{\otimes k}\|_F^2\\
    &= \|x\|^{2k} - 2\inner{x}{y}^k + (\|x\|^2 + 2\gamma + \|x - y\|^2)^k\\
    &= \|x\|^{2k} - 2\left(\|x\|^2 + \gamma\right)^k + \left(\|x\|^2 + 2\gamma + \|x - y\|^2\right)^k\\
    &= \|x\|^{2k} - 2\sum_{j=0}^k {k\choose j}\|x\|^{2(k - j)}\gamma^j + \sum_{j=0}^m {k\choose j}\|x\|^{2(k - j)}(2\gamma + \|x - y\|^2)^j\\
    &= \|x\|^{2k} - 2\left(\|x\|^{2k} + k\|x\|^{2k - 2}\gamma + \sum_{j=2}^k {k\choose j}\|x\|^{2(k - j)}\gamma^j\right)\\
    & \hspace{3em} + \Biggl( \|x\|^{2k} + k\|x\|^{2k - 2}(2\gamma + \|x-y\|^2) \\
    &\hspace{4em} + \sum_{j=2}^k {k\choose j}\|x\|^{2(k - j)}(2\gamma + \|x - y\|^2)^j\Biggr)\\
    &= -2\gamma^2\sum_{j=2}^{k}{k\choose j}\underbrace{\|x\|^{2(k - j)}\gamma^{j-2}}_{\geq -1} + k\underbrace{\|x\|^{2k - 2}}_{\leq 1}\|x - y\|^2 \\
    & \hspace{3em} + (2\gamma + \|x - y\|^2)^2\sum_{j = 2}^k\underbrace{\|x\|^{2(k - j)}(2\gamma + \|x - y\|^2)^{j - k}}_{\leq 1}\\
    &\leq 2\gamma^2\cdot 2^k + k\|x-y\|^2 + \left(4\gamma^2 + 4\gamma\|x - y\|^2 + \|x - y\|^4\right)2^k\\
    &\leq(11\cdot 2^k + k)\|x-y\|^2\\
    &\leq 12\cdot 2^k\|x - y\|^2. 
  \end{align*}
  In the third line from the bottom we are using the fact that the sum of the binomial coefficients is equal to $2^m$.
\end{proof}

\begin{defn}
  Let $p$ and $q$ be the densities of two measures $P$ and $Q$ that are absolutely continuous with respect to the Lebesgue measure. The $\chi^2$ divergence, $\KL$ divergence, and $\TV$ distance are defined as 
  \begin{align*}
    \chi^2(P\|Q) &= \int\frac{(p(x) - q(x))^2}{q(x)}\,dx\\
    \KL(P\|Q) &= \int p(x)\log\left(\frac{p(x)}{q(x)}\right)\,dx\\
    \TV(P,Q) &= \sup_A|P(A) - Q(A)| = \sup\left|\int_A (p(x) - q(x))\,dx\right|.
  \end{align*}
  In the definition of the $\TV$ norm, the supremum is over all measurable sets.
\end{defn}

\begin{lem}
  \label{lem:KL-leq-chi2}
  We have the bound 
  \begin{equation}
    \KL(P\|Q)\leq \chi^2(P\|Q).
  \end{equation}
\end{lem}
\begin{proof}
  This is an immediate consequence of the fact that $\log(x)\leq x - 1$.
\end{proof}

\begin{lem}[Pinsker's Inequality, Lemma 2.5 in \cite{Tsy}]
  \label{lem:pinsker}
  For any two measures $P$ and $Q$, 
  \begin{equation}
    \TV(P, Q)\leq \sqrt{\KL(P\|Q) / 2}.
  \end{equation}
\end{lem}

\begin{lem}
  \label{lem:KLofprod}
  Let $P$ and $Q$ be random variables and $P^{\otimes N}$ and $Q^{\otimes N}$ denote the distribution of $N$ independent draws from $P$ and $Q$. Then $\KL(P^{\otimes N}\|Q^{\otimes N}) = N\KL(P\|Q)$.
\end{lem}
\begin{proof}
  For notational clarity, we only prove the case where $N = 2$. Then 
  \begin{align*}
    \KL(P^{\otimes 2}\|Q^{\otimes 2}) &= \int_Y\int_X p(x)p(y)\log\left(\frac{p(x)p(y)}{q(x)q(y)}\right)\,dxdy\\
    &= \int_Yp(y)\int_X p(x)\left[\log\left(\frac{p(x)}{q(x)}\right) + \log\left(\frac{p(y)}{q(y)}\right)\right]\,dxdy\\
    &= \int_Y p(y)\int_X p(x)\log\left(\frac{p(x)}{q(x)}\right)\,dxdy + \int_Y p(y)\int_X p(x)\log\left(\frac{p(y)}{q(y)}\right)\,dxdy\\
    &= \int_X p(x)\log\left(\frac{p(x)}{q(x)}\right)\,dx + \int_Y p(y)\log\left(\frac{p(y)}{q(y)}\right)\,dy\\
    &= 2\KL(P\|Q).
  \end{align*}
\end{proof}

\bibliographystyle{plain} 
\bibliography{references}
\end{document}